\documentclass[12pt]{article}

\usepackage{amsthm,amsmath,amssymb}

\usepackage{graphicx}

\usepackage[colorlinks=true,citecolor=black,linkcolor=black,urlcolor=blue]{hyperref}

\usepackage{xcolor}

\usepackage[utf8]{inputenc}

\usepackage{comment}

\definecolor{darkgreen}{RGB}{20,200,10}
\newcommand\nz[1]{\mbox{}
{\marginpar{\color{darkgreen}N}}
{\sf\noindent\color{darkgreen}#1}}%
\newcommand\jc[1]{\mbox{}
{\marginpar{\color{blue}J}}
{\sf\noindent\color{blue}#1}}%

\theoremstyle{plain}
\newtheorem{theorem}{Theorem}
\newtheorem{lemma}[theorem]{Lemma}
\newtheorem{corollary}[theorem]{Corollary}
\newtheorem{proposition}[theorem]{Proposition}

\theoremstyle{definition}
\newtheorem{definition}[theorem]{Definition}
\newtheorem{example}[theorem]{Example}

\newtheorem{question}[theorem]{Question}

\theoremstyle{remark}

\graphicspath{{./Images/}}

\DeclareMathOperator{\Rootmap}{RootIns}
\DeclareMathOperator{\Rootdiag}{RootIns}
\DeclareMathOperator{\Mapins}{MapIns}
\DeclareMathOperator{\Diagins}{DiagIns}

\newcommand{\RootEdge}[1]{\Rootmap_{#1}}
\newcommand{\RootChord}[1]{\Rootdiag_{#1}}
\newcommand{\MapIns}[2]{\Mapins_{#1,#2}}
\newcommand{\DiagIns}[2]{\Diagins_{#1,#2}}

\DeclareMathOperator{\varbox}{\oplus}

\DeclareMathOperator{\rid}{RootInDeg}
\DeclareMathOperator{\id}{InDeg}
\DeclareMathOperator{\ve}{v}

\def\multiset#1#2{\ensuremath{\left(\kern-.3em\left(\genfrac{}{}{0pt}{}{#1}{#2}\right)\kern-.3em\right)}}

\newcommand\chd[2]{(#1,#2)}

\newcommand\definand[1]{{\emph{#1}}}

\DeclareMathOperator{\orbit}{orbit}

\title{Connected chord diagrams and bridgeless maps}

\author{Julien Courtiel\thanks{Supported by the French ``Agence Nationale de la Recherche'' MetAConC.}\\
\small Universit\'e de Caen Normandie\\[-0.8ex]
\small\tt julien.courtiel@unicaen.fr\\
\and
 Karen Yeats\thanks{Supported by an NSERC Discovery grant.} \\
\small University of Waterloo\\[-0.8ex]
\small\tt kayeats@uwaterloo.ca \\
\and Noam Zeilberger\\
\small University of Birmingham\\[-0.8ex]
\small\tt zeilbern@cs.bham.ac.uk
}

\date{\,}

\begin{document}

\maketitle

\begin{abstract}{We present a surprisingly new connection between two well-studied combinatorial classes: rooted connected chord diagrams on one hand, and rooted bridgeless combinatorial maps on the other hand. We describe a bijection between these two classes, which naturally extends to indecomposable diagrams and general rooted maps. As an application, this bijection provides a simplifying framework for some technical quantum field theory work realized by some of the authors.
Most notably, an important but technical parameter naturally translates to vertices at the level of maps. We also give a combinatorial proof to a formula which previously resulted from a technical recurrence, and with similar ideas we prove a conjecture of Hihn.
Independently, we revisit an equation due to Arquès and Béraud for the generating function counting rooted maps with respect to edges and vertices, giving a new bijective interpretation of this equation directly on indecomposable chord diagrams, which moreover can be specialized to connected diagrams and refined to incorporate the number of crossings.
Finally, we explain how these results have a simple application to the combinatorics of lambda calculus, verifying the conjecture that a certain natural family of lambda terms is equinumerous with bridgeless maps.
}
\end{abstract}

%----------------------
\section{Introduction}
\label{sec:introduction}

Connected chord diagrams are well-studied combinatorial objects that appear in numerous mathematical areas such as 
knot theory \cite{Stoimenow,Bori2000,Zagier2001}, graph sampling \cite{Acan2013}, analysis of data structures \cite{flajolet-datastructure}, and  bioinformatics \cite{combRNA}.
Their counting sequence (Sloane's \href{https://oeis.org/A000699}{A000699})
%(listed in the On-Line Encyclopedia of Integer Sequences as \href{https://oeis.org/A000699}{A000699})
has been known since Touchard's early work \cite{Touchard1952}.
In this paper we present a bijection with another fundamental class of objects: bridgeless combinatorial maps.
Despite the ubiquity of both families of objects in the literature, this bijection is, to our knowledge, new. Furthermore, it is fruitful in the sense that it generalizes and restricts well, and useful parameters carry through it.

%As an application, this bijection gives a simpler framework of some technical quantum field theory work realized by some of the authors~\cite{MYchord, HYchord, CYchord}. Most notably, an important but obscure technical parameter naturally translates to vertices in maps.We then

%There is a tendency to think that the descriptions of the combinatorial integer sequences occurring in the OEIS (\textit{On-Line Encyclopedia of Integer Sequences}) are given in terms of the simplest and most straightforward  objects  possible --- the website has proved to be very popular and exhaustive.

%Let us look at the central integer sequence of the present article. It starts with $1, 1, 4, 27, 248, 2830, 38232, 593859$, and is indexed by \texttt{A000699} in OEIS.  According to what we can read on OEIS at this day (\today), this sequence almost exclusively counts the "\textit{irreducible diagrams with $2n$ nodes}". 

%%%%%%%%%%%%%%%%%%%%%%%%
\subsection{Definitions}
\label{sec:defn}
%%%%%%%%%%%%%%%%%%%%%%%%

Before outlining the contributions of the paper more precisely, we begin by recalling here the formal definitions of (rooted) chord diagrams and (rooted) combinatorial maps, together with some auxiliary notions and notation.

%%%%%%%%%%%%%%%%%%%%%%%%
\subsubsection{Chord diagrams}
\label{sec:defn:diagrams}
%%%%%%%%%%%%%%%%%%%%%%%%

%(cf.~\cite[\S1.3, \S6.1]{LZgraphs}),
% as well as the different ways of representing them (up to isomorphism).
% by elementary combinatorial objects.

\begin{definition}[Matchings on linear orders]
Let $P$ be a linearly ordered finite set.
An \definand{$n$-matching} in $P$ is a mutually disjoint collection $C$ of ordered pairs $(a_1,b_1), \dots, (a_n,b_n)$ of elements of $P$, where $a_i < b_i$ for each $1 \le i \le n$.
A \definand{perfect matching} in $P$ is a matching which includes every element of $P$.
\end{definition}
% \noindent
% Equivalently, a (perfect) $n$-matching in $P$ can be thought of as a monotone, injective (and surjective) function
% $$C : \underbrace{2 + \dots + 2}_{n\text{ times}} \to P$$
% from $n$ copies of the interval $2 = \{0 < 1\}$ into $P$.
% \jc{I'm not sure this definition is very clear to me. It confuses me more than anything else. I would stick to the first definition, which is great.}
% \nz{Fair enough. It was a useful reformulation for me, but it's true that it isn't used at all in the paper and doesn't really add anything. So if you find it confusing, better to just remove it.}

\begin{definition}[Chord diagrams]
A \definand{rooted chord diagram} is a linearly ordered, non-empty finite set $P$ equipped with a perfect matching $C$.
%A pair of successive points in $P$ is called an \definand{interval}, while
The pairs in $C$ are called \definand{chords}, while the \definand{root chord} is the unique pair whose first component is the least element of $P$.
\end{definition}
% \begin{definition}[Chord diagram] A \emph{fixed point free involution}  is a permutation only composed of $2$-cycles. A \emph{chord diagram} with $n$ chords is a graphical representation of a fixed point free involution on the set $\{0,1,2,\ldots, 2n-1\}$, in which we draw points $0,1,2,\ldots,2n-1$ on a horizontal line from left to right, and we join by a chord every two elements belonging to a same cycle. The \emph{root chord} of a chord diagram is the chord including the point $0$.
% \end{definition}
\noindent
Two $n$-matchings $(P,C)$ and $(P',C')$ are considered isomorphic if they are equivalent up to relabeling of the elements and reordering of the pairs, or in other words, if there is an order isomorphism $\phi : P \cong P'$ and a permutation $\pi \in S_n$ such that $\phi C = C' \pi$, where $\phi C = (\phi(a_1),\phi(b_1)), \dots, (\phi(a_n),\phi(b_n))$ denotes the image of $C$ % $C = (a_1,b_1),\dots,(a_n,b_n)$
under $\phi$, and $C' \pi = (a'_{\pi(1)},b'_{\pi(1)}),\dots,(a'_{\pi(n)},b'_{\pi(n)})$ denotes the reindexing of $C'$ % $C' = (a'_1,b'_1),\dots,(a'_n,b'_n)$
by $\pi$.
Up to isomorphism, a chord diagram with $n\ge 1$ chords may therefore be identified with a perfect matching on the ordinal $2n = \{0 < \dots < 2n-1\}$, and so we will usually omit reference to the underlying set of a chord diagram, simply keeping track of the number of chords $n$ (we refer to the latter as the \definand{size} of the diagram).
Isomorphism classes of chord diagrams of size $n$ can also be presented as \emph{fixed point-free involutions} on the set $2n$, although we find the definition as a perfect matching more convenient to work with.
% any fixed point-free involution $f$ determines a perfect matching $C$ by $C_f = \{(a,f(a)) \mid a < f(a)\}$, and conversely any perfect matching $C$ determines a fixed point-free involution $f_C$ whose graph is $C \cup C^{\text{sym}}$.

% For example, the fixed point free involution $(0\,3)(1\,5)(2\,4)$ can be represented by the diagram on the left in Figure~\ref{fig:intermsofpermu}.
To visualize a chord diagram, we represent the elements of its underlying linear order by a series of collinear dots, and the matching by a collection of arches joining the dots together in pairs: see Figure~\ref{fig:intermsofpermu}(a) for an example.
\begin{figure}%[h!]
\begin{center}
 \hfill
 (a) 
\begin{minipage}[c]{0.25\textwidth}
\includegraphics[scale=3.3]{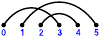} 
\end{minipage}
 \hfill \hfill (b)
\begin{minipage}[c]{0.15\textwidth}
\includegraphics[scale=2]{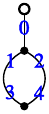}
\end{minipage}
\begin{minipage}[c]{0.3\textwidth}
% would use an align* environment here, except that it apparently does
% not interact well with minipage (https://tex.stackexchange.com/questions/36954/spurious-space-above-align-environment-at-top-of-page-minipage)
\begin{tabular}{r@{ }c@{ }l}
$H$ & =& $\{\,0,1,2,3,4\,\}$ \\
$\sigma$ &=& $(0\,1\,2)(3\,4)$ \\
$\alpha$ &=& $(0)(1\,3)(2\,4)$
\end{tabular}
\end{minipage}
 \hfill \,
\end{center}
\caption{ (a) Rooted chord diagram associated to the perfect matching $\chd03, \chd15,\chd24$. (b) A rooted map and its permutation representation.}
\label{fig:intermsofpermu}
\end{figure}
In the literature, rooted chord diagrams are also drawn according to a circular convention: instead of being arranged on a line, the $2n$ points are drawn on an oriented circle and joined together by chords, and then one point is marked as the root.
This convention has been notably used in~\cite{MYchord,HYchord}, but the linear convention is the one we adopt for the rest of the document\footnote{People also consider \emph{unrooted} chord diagrams with no marked point, see for example \cite[\S6.1]{LZgraphs}. Since we work only with rooted chord diagrams in this paper, we refer to them simply as chord diagrams, or even as ``diagrams'' when there is no confusion.}.

%\begin{definition}[Chord diagram]
%A \definand{chord diagram} of size $n$ is an oriented circle with $n$ distinguished pairs of points joined by chords, considered up to orientation-preserving diffeomorphism of the circle.
%\end{definition}
\begin{definition}[Intersection graph, connected diagrams]
\label{def:intgraph-condiag}
The \definand{intersection graph} of a chord diagram $C$ is defined as the digraph with a vertex for every chord, and an oriented edge from chord $(a,b)$ to chord $(c,d)$ whenever $a < c < b < d$.
A chord diagram is said to be \definand{connected} (or \emph{irreducible}) if its intersection graph is (weakly) connected.
\end{definition}
\noindent
Equivalently, a diagram of size $n$ is connected if for every proper non-empty subsegment $[i,j] \subset [0,2n-1]$, there exists a chord with one endpoint in $[i,j]$ and the other endpoint outside $[i,j]$.
All connected diagrams of size $\le 3$ are depicted in the first row of Table~\ref{tab:smallex}.

%\noindent
%\begin{definition}
%A \definand{rooting} of a chord diagram is a choice of basepoint on the circle, distinct from the endpoints of any chord.
%A \definand{rooted chord diagram} is a chord diagram equipped with a rooting, considered up to orientation-and-basepoint-preserving diffeomorphism of the circle.
%\end{definition}

% \jc{Should we really adopt this "rooted" definition? Maybe Karen you want to have this to align with your other papers?}
% \ky{I'm happy if we use the linear matching as the definition since that is both how we draw them and how we use them, but then I would want a paragraph pointing out that you could alternately use the oriented rooted circle definition and so making the connection with my previous stuff (and so also explaining the language of ``chord diagram'').  Would that suit you?}
% \jc{Are you OK with what I've done?}

%Since we will be dealing strictly with rooted chord diagrams in this paper, we sometimes refer to them as chord diagrams when there is no confusion, or even simply as \emph{diagrams}.
%A convenient way of visualizing a rooted chord diagram is to ``cut and unfold'' the circle at the chosen basepoint, leaving a linear arrangement of points joined in pairs by a set of (potentially crossing) arches.

\begin{table}%[h!]
\begin{center}
\begin{tabular}{cccc} \hline
Objects &  Size $1$ & Size $2$ & Size $3$ \\
\hline
\begin{minipage}{0.14\textwidth}
\vspace*{0.1cm}
\begin{center}
Connected diagrams 
\vspace*{0.1cm}
\end{center}
\end{minipage}
& 
\begin{minipage}{0.15\textwidth}
\begin{center}
\includegraphics[scale=1.3]{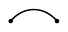} 
\end{center}
\end{minipage}
& 
\begin{minipage}{0.15\textwidth}
\begin{center}
\includegraphics[scale=1.3]{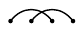} 
\end{center}
\end{minipage}
&
\begin{minipage}{0.4\textwidth}
\begin{center}
\includegraphics[width=\textwidth]{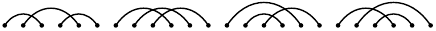} 
\end{center}
\end{minipage} 
 \\ \hline
\begin{minipage}{0.14\textwidth}
\begin{center}
Bridgeless maps
\end{center}
\end{minipage}
& 
\begin{minipage}{0.15\textwidth}
\begin{center}
\includegraphics[scale=1.3]{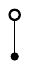} 
\end{center}
\end{minipage}
& 
\begin{minipage}{0.15\textwidth}
\begin{center}
\includegraphics[scale=1]{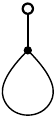} 
\end{center}
\end{minipage}
&
\begin{minipage}{0.4\textwidth}
\begin{center}
\includegraphics[width=\textwidth]{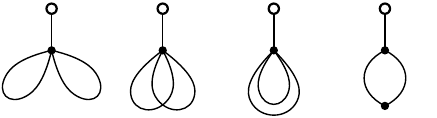} 
\end{center}
\end{minipage} 
\\
\hline 
\end{tabular}
\end{center}
\caption{Small connected diagrams and bridgeless maps}
\label{tab:smallex}
\end{table}
\noindent
Besides connectedness, we also consider the weaker notion of ``indecomposability'' of a diagram, defined in terms of diagram concatenation.
\begin{definition}[Diagram concatenation]
Let $C_1$ and $C_2$ be chord diagrams of sizes $n_1$ and $n_2$, respectively.
The \definand{concatenation} of $C_1$ and $C_2$ is the chord diagram $C_1C_2$ of size $n_1 + n_2$ whose underlying linear order is given by the ordinal sum of the underlying linear orders of $C_1$ and $C_2$, and whose matching is determined by $C_1$ on the first $2n_1$ elements and by $C_2$ on the next $2n_2$ elements.
\end{definition}
\noindent
As the name suggests, diagram concatenation has a simple visual interpretation as laying two chord diagrams side by side.
\begin{definition}[Indecomposable diagrams]
A rooted chord diagram is said to be \definand{indecomposable} if it cannot be expressed as the concatenation of two smaller diagrams.
\end{definition}
\noindent
Every connected diagram is indecomposable, but the converse is not true: see Table~\ref{tab:smallex2}.
\begin{table}%[h!]
\begin{center}
\begin{tabular}{ccc} \hline
Objects &  Size $2$ & Size $3$ \\
\hline
\begin{minipage}{0.215\textwidth}
\begin{center}
\vspace*{0.1cm}
Indecomposable disconnected diagrams
\vspace*{0.1cm}
\end{center}
\end{minipage}
& 
\begin{minipage}{0.07\textwidth}
\begin{center}
\includegraphics[scale=1.3]{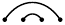} 
\end{center}
\end{minipage}
&
\begin{minipage}{0.6\textwidth}
\begin{center}
\includegraphics[width=\textwidth]{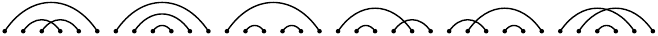} 
\end{center}
\end{minipage} 
 \\ \hline
\begin{minipage}{0.2\textwidth}
\begin{center}
 Maps with at least one bridge
\end{center}
\end{minipage}
& 
\begin{minipage}{0.07\textwidth}
\begin{center}
\includegraphics[scale=1.3]{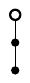} 
\end{center}
\end{minipage}
&
\begin{minipage}{0.6\textwidth}
\begin{center}
\includegraphics[width=0.95\textwidth]{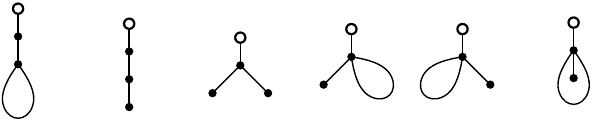} 
\end{center}
\end{minipage} 
\\
\hline 
\end{tabular}
\end{center}
\caption{Small indecomposable diagrams and maps not displayed in Table~\ref{tab:smallex}.}
\label{tab:smallex2}
\end{table}

Finally, it will often be convenient for us to speak about \emph{intervals} in a chord diagram.
By an interval, we simply mean a pair of successive points: thus a diagram with $n$ chords (joining $2n$ points) has $2n-1$ intervals.

% The \emph{concatenation} of two rooted chord diagrams can be defined in the evident way in terms of the linear representation.
% This leads to a notion of ``indecomposability'' for rooted chord diagrams, which is weaker than connectivity.
% \begin{definition}[Indecomposable chord diagrams]
% A rooted chord diagram is said to be \definand{indecomposable} if it cannot be expressed as the concatenation of two smaller diagrams, or equivalently, if no set of chords can be separated from the remaining chords by a vertical line.
% \end{definition}
% \noindent
% Every connected diagram is obviously indecomposable, but the converse is not true (cf.~Table~\ref{tab:smallex2}).

%%%%%%%%%%%%%%%%%%%%%%%%
\subsubsection{Combinatorial maps}
\label{sec:defn:maps}
%%%%%%%%%%%%%%%%%%%%%%%%

Combinatorial maps are representations of embeddings of graphs into oriented surfaces \cite{JSmaps,LZgraphs,Eynard2016}.
Like chord diagrams, they come in both rooted and unrooted versions, but we will be dealing only with rooted maps in this paper.
\begin{definition}[Combinatorial maps]
A \definand{rooted combinatorial map} is a transitive permutation representation of the group $\Gamma = \langle \sigma,\alpha \mid \alpha^2 = 1\rangle$, equipped with a distinguished fixed point for the action of $\alpha$.
Explicitly, this consists of the following data:
\begin{itemize}
\item a set $H$ (whose elements are called \definand{half-edges});
\item a permutation $\sigma$ and an involution $\alpha$ on $H$;
\item a half-edge $r \in H$ (called the \definand{root}) for which $\alpha(r) = r$;
\item such that between any pair of half-edges $x,y \in H$, there is a permutation $f$ defined using only compositions of $\sigma$ and $\alpha$ (and/or their inverses) for which $f(x) = y$.
\end{itemize}
\end{definition}
\noindent
Two rooted combinatorial maps are considered isomorphic just when there is a bijection between their underlying sets of half-edges which commutes with the action of $\Gamma$ and preserves the root.
Note that our definition of combinatorial maps is a bit non-standard in allowing the involution $\alpha$ to contain fixed points and taking the root as a distinguished fixed point of $\alpha$.
Defining the root as a fixed point is convenient for dealing with the trivial map (pictured at the left end of the second row of Table~\ref{tab:smallex}), while the presence of additional fixed points means that in general our maps can have ``dangling edges'' in addition to the root.
Formally, the underlying graph of a combinatorial map is defined as follows.
\begin{definition}[Underlying graph]
Let $M = (H,\sigma,\alpha,r)$ be a rooted combinatorial map.
The \definand{underlying graph} of $M$ has vertices given by the orbits of $\sigma$, edges given by the orbits of $\alpha$, and the incidence relation between vertices and edges defined by their intersection.
% \begin{tikzcd}
% H\ar[r,"\sigma-\text{orbit}"]\ar[loop,"\alpha",swap,in=130,out=60,looseness=3] & V
% \end{tikzcd}
\end{definition}
\noindent
For any $v \in \orbit(\sigma)$ and $e \in \orbit(\alpha)$ we have $|v \cap e| \in \{ 0,1,2 \}$, that is, a vertex and an edge can be incident either zero, once, or twice in the underlying graph.
An edge which is incident to the same vertex twice is called a \definand{loop}, while an edge which is incident to only one vertex exactly once is called a \definand{dangling edge}.
The \definand{size} of a map is defined here as the number of edges in its underlying graph (giving full value to dangling edges).
We call a combinatorial map \definand{closed} if its underlying graph contains no dangling edges other than the root, and otherwise we call it \definand{open}.
For the most part, we will be dealing with closed maps, so we usually omit the qualifier unless it is important to remind the reader when we are dealing with open maps (as will at times be convenient).
We also usually omit the prefix ``rooted'', again because we only ever consider rooted combinatorial maps.

Figure~\ref{fig:intermsofpermu}(b) shows an example of a (closed rooted) combinatorial map and its graphical realization, where we have indicated the unattached end of the root by a white vertex.
This is also an example of a \emph{bridgeless} map in the sense of the definition below.
\begin{proposition}
The underlying graph of any combinatorial map is connected.
\end{proposition}
\begin{proof}
By transitivity of the action of $\Gamma$.
\end{proof}
\begin{definition}[Bridgeless maps]
A combinatorial map is said to be \definand{bridgeless} if its underlying graph is 2-edge-connected, that is, if there does not exist an edge whose deletion separates the graph into two connected components (such an edge is called a \definand{bridge}).
\end{definition}
\noindent
The second row of Table~\ref{tab:smallex} lists all (closed) bridgeless maps with at most three edges, while the second row of Table~\ref{tab:smallex2} lists all the remaining maps of size $\le 3$.
Observe that although the half-edges are unlabeled (again, since we are interested in isomorphism classes of labelled structures), the specification of the permutation $\sigma$ is contained implicitly in the cyclic ordering of the half-edges around each vertex, and the specification of the involution $\alpha$ in the gluing together of half-edges to form edges.
Observe also that one of the maps in Table~\ref{tab:smallex} contains a pair of crossing edges: such crossings should be thought of as ``virtual'', arising from the projection of a graph embedded in a surface of higher genus down to the plane.
For a more detailed discussion of the precise correspondence between combinatorial maps and embeddings of graphs into oriented surfaces, see \cite{JSmaps,LZgraphs,Eynard2016}.

Finally, we introduce a few additional technical notions.
In a rooted map, we distinguish the root from the \emph{root edge} and the \emph{root vertex}: the root vertex is the unique vertex which is incident to the root, while the root edge (in a map of size $> 1$) is the unique edge following the root in the positive direction (i.e., according to the permutation $\sigma$) around the root vertex.
% Unlike other papers, the root here counts as an edge. 
A \textit{corner} is the angular section between two distinct adjacent half-edges.
The \textit{root corner} is the corner between the root and the root edge.
Half-edges are in obvious bijection with corners (for maps of size $> 1$), but it is often more convenient to work with the corners: for example, pointing out two corners is a clear way to show how to insert an edge in a map.

% For chord diagrams, the analogue of a corner is an \emph{interval}, defined as the segment of the circle between two successive marked points, or the gap between a pair of adjacent points in the linear representation.
%% Given a chord diagram $C$ with $n$ chords, let the $2n$ dots on the line underlying some drawing of $C$ be $v_1, \ldots, v_{2n}$ in order.  Then the \emph{intervals} of $C$ are the intervals $(v_1, v_2), (v_2, v_3), \ldots, (v_{2n-1}, v_{2n})$ along the line.  The analogous notion for maps is the notion of corner, defined above.

%%%%%%%%%%%%%%%%%%%%%%%
\subsection{Enumerative and bijective links between maps and diagrams}
%%%%%%%%%%%%%%%%%%%%%%%

We demonstrate in this paper the existence of a size-preserving bijection between bridgeless maps and connected diagrams:
\[
 [\text{bridgeless combinatorial maps}] \overset{\theta}\longleftrightarrow
 [\text{connected chord diagrams}].
\]
Indeed, we prove that $\theta$ is the \emph{restriction} of a bijection between combinatorial maps and indecomposable diagrams:
\[
 [\text{combinatorial maps}] \overset{\phi}\longleftrightarrow
 [\text{indecomposable chord diagrams}].
\]
Conversely, we prove that $\phi$ is the \emph{extension} of $\theta$ obtained by composing with a canonical decomposition of rooted maps (respectively, indecomposable diagrams) in terms of the bridgeless (respectively, connected) component of the root.

The existence of $\theta$ implies in particular the following enumerative statement.
\begin{theorem} The number of rooted bridgeless combinatorial maps of size $n$ is equal to the number of rooted connected chord diagrams of size $n$.
\label{theo:simplebijection}
\end{theorem}
\noindent
The fact that bridgeless maps and connected diagrams define equivalent combinatorial classes has apparently not been previously observed in the literature, let alone with a bijective proof.
In contrast, an explicit bijection between combinatorial maps and indecomposable diagrams was already given by Ossona de Mendez and Rosenstiehl \cite{OdMRtrans,OdMRencoding}, who moreover wrote (in the early 2000s) that the corresponding enumerative statement ``was known for years, in particular in quantum physics'', although ``no bijective proof of this numerical equivalence was known''.
\begin{theorem}[Ossona de Mendez and Rosenstiehl \cite{OdMRtrans,OdMRencoding}]
The number of rooted combinatorial maps of size $n$ is equal to the number of rooted indecomposable chord diagrams of size $n$.
\label{theo:OMRbijection}
\end{theorem}
\noindent
It may appear surprising that Theorem~\ref{theo:simplebijection} has been seemingly overlooked despite Theorem~\ref{theo:OMRbijection} having been ``known for years'', and with the latter even being given a nice bijective proof over a decade ago (that was further analyzed and simplified by Cori \cite{Cori2009}).
Yet, as we will discuss, there is a partial explanation, namely that Ossona de Mendez and Rosenstiehl's bijection \emph{does not restrict} to a bijection between bridgeless maps and connected maps (and moreover cannot for intrinsic reasons, see Section~\ref{subsec:arques-beraud-statement}).
In other words, both of the bijections $\theta$ and $\phi$ we describe in this paper are apparently fundamentally new.

\subsection{Structure of the document}
%%%%%%%%%%%%%%%%%%%%%%%

We will begin in Section~\ref{sec:cardinalities} by showing that connected diagrams and bridgeless maps are equinumerous due to them satisfying the same recurrences, and similarly for indecomposable diagrams and general maps.  Implicitly this already induces bijections, but there are choices to be made, and good choices will give bijections preserving interesting and important statistics.  Thus we will proceed in Section~\ref{sec:operations} to define operations on diagrams and maps which will be the building blocks of the bijections.  The bijections themselves are presented in Section~\ref{sec:bijection}.  Our bijection from connected diagrams to bridgeless maps has two descriptions, one of which makes clear that it extends to a bijection between indecomposable diagrams and general maps that we also give.  Furthermore, we characterize those diagrams which are taken to planar maps under our bijection.

The remainder of the paper looks at applications resulting from our bijections.
Section~\ref{sec:qft} applies our bijection from connected diagrams to some chord diagram expansions in quantum field theory which some of us, with other collaborators, have discovered as series solutions to a class of functional equations in quantum field theory.  Some interesting results have been proved thanks to the diagram expansions, but some of the diagram parameters were obscure.  We will use our bijections to maps  to simplify and make more natural these parameters and the resulting expansion.  Most notably, a special class of chords, known as terminal chords, corresponds to vertices in the maps. Moreover, we use this new interpretation in terms of maps to give a combinatorial proof to a quite involved formula appearing in~\cite{HYchord}, which was a key point of that article but did not have a clear explanation aside a technical recurrence, and with similar ideas we prove a conjecture of Hihn.

Section~\ref{sec:arques-beraud} revisits a functional equation of Arquès and Béraud for the generating function counting rooted maps with respect to edges and vertices. We give a new bijective interpretation of this functional equation directly on indecomposable chord diagrams, with the important property that it restricts to connected diagrams to verify a modified functional equation. These equations have also appeared recently in studies of the combinatorics of lambda calculus, and we explain how to use our results to verify a conjecture that a certain family of lambda terms is equinumerous with bridgeless maps.

%This functional equation originally was used for counting rooted maps with respect to edges and vertices, but it also counts certain isomorphism classes of terms in lambda calculus.
%We reinterpret the equation using our bijections, the benefit being that we can both extend to also counting crossings and restricts to only connected diagrams.

%%%%%%%%%%%%%%%%%%%%%%%
\section{Equality of the cardinality sequences}
\label{sec:cardinalities}
%%%%%%%%%%%%%%%%%%%%%%%

Once the observation has been made, it is quite elementary to show that the cardinalities of the above-mentioned classes are the same by proving that they satisfy the same recurrences, as we will do in this section.
First, we establish the recurrence for connected diagrams and bridgeless maps, which implies Theorem~\ref{theo:simplebijection}.
Then, we establish a recurrence for indecomposable diagrams and unrestrained maps, which yields a new proof of Theorem~\ref{theo:OMRbijection}.
% the result of Ossona de Mendez and Rosenstiehl \cite{OdMRencoding}.  % and Cori \cite{Cori2009}
Note that the propositions we prove in this section also yield implicit correspondences between the combinatorial classes, but they do not determine which map a given diagram must be sent to.
Although it is easy to settle that in an arbitrary way, the more careful analysis of Section~\ref{sec:operations} and \ref{sec:bijection} will yield bijections preserving various important statistics.

\subsection{Between connected diagrams and bridgeless maps}
\label{ss:btwn-cd-bm}

We combinatorially show the following recurrence -- which characterizes the sequence \texttt{A000699} in the OEIS -- for connected diagrams and bridgeless maps.
%
%
%Sadly no clear explicit correspondence seems to exist between connected diagrams and bridgeless maps. The approach we use consists in proving that both sets of objects share the same recursion. \ky{I don't like how this is said.  The matching recursive constructions induce a bijection.  I realize this is not fully satisfactory since we have no non-recursive description of it, but it is fully explicit.}
%\jc{OK I will go back to this later -- maybe when this section is finished.}

\begin{proposition} The number $c_n$ of rooted connected diagrams of size $n$ and the number of rooted bridgeless maps of size $n$ both satisfy $c_1 =1$ and
\begin{equation}
c_n = \sum_{k=1}^{n-1} (2k-1) \, c_k \, c_{n-k}.
\end{equation}
\label{prop:recurrence}
\end{proposition}

\begin{proof} The recurrence relation translates the fact that it is possible to combine two objects, one of which is weighted by twice its size (minus 1), to bijectively give a bigger object of cumulated size.   We describe how to do so for our two classes. \vspace{3pt}

\begin{figure}[!ht]
\begin{center}
\begin{tabular}{c}
{ \begin{minipage}{0.13\textwidth}
\small
\begin{center}
 connected diagrams
\end{center} \end{minipage}
} \ \,\ \ \  \begin{minipage}{0.7\textwidth}
\begin{center}
\includegraphics[width=0.97\textwidth]{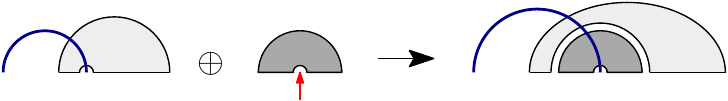} 
\end{center}
\end{minipage} \ \ \ 
\\ 
{ \begin{minipage}{0.12\textwidth}
\small
\begin{center}
 bridgeless maps
\end{center} \end{minipage}
} \ \ \  \ \ $\left\{\ \ \begin{minipage}{0.7\textwidth}
\begin{center}
\includegraphics[width=0.95\textwidth]{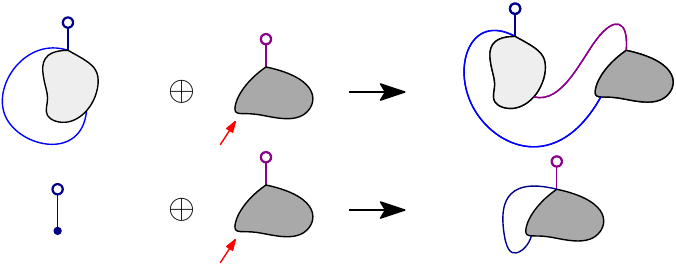} 
\end{center}
\end{minipage}
\right.$
\end{tabular}
\end{center}
\caption{Schematic decomposition of connected diagrams and bridgeless maps.}
\label{fig:decomp}
\end{figure}

\noindent \textbf{Connected diagrams.}
For connected diagrams, $2k-1$ counts the number of intervals delimited by $k$ chords.
%\nz{It seems that we no longer have an explicit definition of ``interval'' in the current version of the document. Is it obvious, or should we put the definition somewhere in Section~\ref{sec:defn:diagrams}?}
%\nz{Okay, I put an inline definition of intervals back in, at the end of Section~\ref{sec:defn:diagrams}.}
In other words, it means there are $2k-1$ ways to insert a new root chord in a diagram of size $k$. We can find in the literature numerous ways to combine a diagram $C_1$ with another diagram $C_2$ with a marked interval~\cite{NWchord}. The one we choose comes from~\cite{CYchord} and is illustrated in Figure~\ref{fig:decomp}. The idea is to insert $C_2$ into $C_1$, just after the root chord of $C_1$. Then, we move the right endpoint of the root chord of $C_1$ to the marked interval of $C_2$. We thus obtain our final combined diagram.

To recover $C_1$ and $C_2$, we mark the interval just after the root chord. Then, we pull the right endpoint of the diagram to the left until the diagram disconnects into two connected components. The first component is $C_1$, the second one $C_2$. \vspace{3pt}

\noindent \textbf{Bridgeless maps.}
In maps of size $k$, the number $2k-1$ refers to the number of corners. Given two maps $M_1$ and $M_2$ where $M_2$ has a marked corner, we construct a larger map as follows (this is also illustrated in Figure~\ref{fig:decomp}).

If $M_1$ has size $1$, we insert a new edge in $M_2$ which links the root corner of $M_2$ to its marked corner. If $M_1$ has size greater than $1$ then it has a root edge. Let us unstick the second endpoint of the root edge and insert it in the marked corner of $M_2$. Then, we take the root of $M_2$ and  insert it where  the second endpoint of the root edge of $M_1$ was. We thus obtain our final map. Note that no bridge has been created in the process.

To recover $M_1$ and $M_2$, we start by marking the corner after the second endpoint of the root edge of the new map. Then, grab this endpoint and slide it up, towards the root. When a bridge appears, we stop the process and cut the bridge, marking it as a root. The two resulting diagrams are $M_1$ are $M_2$. If we reach the root vertex with this process without creating any bridge, then it means that $M_1$ was the trivial map with one half-edge. In that case, we obtain $M_2$ by just removing the root edge.
\end{proof}

\subsection{Between indecomposable diagrams and  maps}
\label{ss:btwn-id-m}

We now prove a similar proposition for indecomposable diagrams and unconstrained maps.
%
%Sadly no clear explicit correspondence seems to exist between connected diagrams and bridgeless maps. The approach we use consists in proving that both sets of objects share the same recursion. \ky{I don't like how this is said.  The matching recursive constructions induce a bijection.  I realize this is not fully satisfactory since we have no non-recursive description of it, but it is fully explicit.}
%\jc{OK I will go back to this later -- maybe when this section is finished.}

\begin{proposition} The number $b_n$ of indecomposable diagrams of size $n$ and the number of rooted  maps of size $n$ both satisfy $b_1 =1$ and
\begin{equation}
b_n = \sum_{k=1}^{n-1} b_k \, b_{n-k} + (2n-3) b_{n-1}.
\label{eq:recind}
\end{equation}
\label{prop:recurrence2}
\end{proposition}

\begin{proof} The decompositions for both classes, which we describe in this proof, are illustrated by Figure~\ref{fig:decomp2}.

\begin{figure}[!ht]
\begin{center}
\begin{tabular}{c}
{ \begin{minipage}{0.2\textwidth}
\small
\begin{center}
indecomposable diagrams
\end{center} \end{minipage}
} \   $\left\{\ \ \begin{minipage}{0.62\textwidth}
\begin{center}
\includegraphics[width=0.95\textwidth]{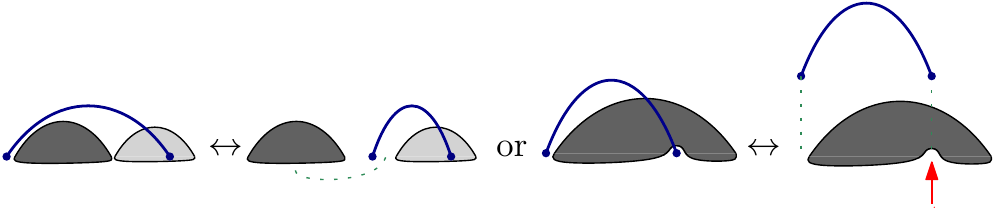} 
\end{center}
\end{minipage}
\right.$ \ 
\\ \ \\
{ \begin{minipage}{0.2\textwidth}
\small
\begin{center}
 rooted maps
\end{center} \end{minipage}
} \   $\left\{\ \ \begin{minipage}{0.62\textwidth}
\begin{center}
\includegraphics[width=0.95\textwidth]{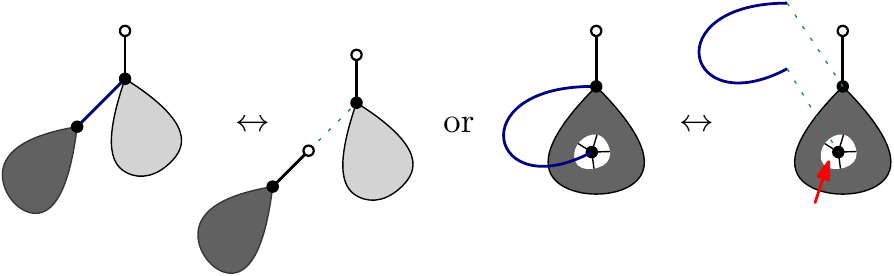} 
\end{center}
\end{minipage}
\right.$\ \ \ \ 
\end{tabular}
\end{center}
\caption{Schematic decomposition of indecomposable diagrams and  maps.}
\label{fig:decomp2}
\end{figure}

\noindent \textbf{Indecomposable diagrams.} For an indecomposable diagram $D$ of size \mbox{$n>1$}, there are two exclusive possibilities.
\begin{itemize}
\item \textbf{The deletion of the root chord makes the diagram decomposable}, i.e. the resulting diagram is the concatenation of several indecomposable diagrams. Let $D_1$ be the first one of them, and $D_2$ the diagram $D$ where we have removed $D_1$ while leaving the root chord in place. The transformation is reversible; we can recover $D$ from $D_1$ and $D_2$ by putting $D_1$ in the leftmost interval (after the left endpoint of the root chord) of $D_2$. Thus, if $D_1$ has size $k$, the number of such diagrams $D$ is $b_k b_{n-k}$.
\item \textbf{The deletion of the root chord induces another indecomposable diagram $D'$.} Then $D'$ has size $n-1$ and we can recover $D$ via a root chord insertion. As mentioned in the proof of Proposition~\ref{prop:recurrence}, a chord diagram with $k$ chords has $2k-1$ intervals, so there are $2n-3$ different ways to insert a root chord in $D'$.
Thus, the number of such diagrams is $(2n-3) b_{n-1}$. 
\end{itemize}
The conjunction of both cases gives Equation~\ref{eq:recind}.

\noindent \textbf{Maps.}
The decomposition we give is based on Tutte's classic root edge removal procedure, extended to the arbitrary genus case \cite{ABmaps,Eynard2016}.
We distinguish again two exclusive possibilities for a rooted  map of size $n>1$.
\begin{itemize}
\item \textbf{The root edge is a bridge.} In other words, $M$ joins two different maps $M_1$ and $M_2$ via a bridge. If $M_1$ has size $k$, there are then $b_k b_{n-k}$ such maps.
\item \textbf{The root edge is not a bridge.} Then $M$ is obtained from a map of size $n-1$ by a root edge insertion. There are $2n-3$ ways to insert a root edge in a map of size $n-1$ (this corresponds to the number of corners). Thus, the number of such maps is $(2n-3) b_{n-1}$. 
\end{itemize}
Again, Equation~\ref{eq:recind} results from the consideration of these two cases.
\end{proof}

%***********************************%
%%%%%%%%%%%%%%%%%%%%%%%%%%%%%%%%%%%%
\section{Basic operations}
\label{sec:operations}
%%%%%%%%%%%%%%%%%%%%%%%%%%%%%%%%%%%%
%***********************************%

We define in this section several basic operations on chord diagrams and combinatorial maps, which will be used in Section~\ref{sec:bijection} to formally construct bijections between connected diagrams and bridgeless maps, and between indecomposable diagrams and general maps.

\subsection{Operations on chord diagrams}

% We begin by defining two natural families of operations on indecomposable chord diagrams.

\begin{definition}[Operations $\Rootdiag$ and $\Diagins$]
Let $D$ be a diagram of size $n$, $k$ an integer $1 \leq k \leq 2n-1$, and $D'$ an arbitrary diagram.
We write $\RootChord k(D)$ to denote the diagram obtained from $D$ by inserting a new root chord whose right endpoint ends in the $k$th interval of $D$ (from left to right), and $\DiagIns {D'} k(D)$ to denote the diagram obtained from $D$ by inserting the diagram $D'$ into the $k$th interval of $D$.  (Figure \ref{fig:diagraminsertion} shows examples of both operations.)
\end{definition}

\begin{figure}[!ht]
\centering
\includegraphics[width = \textwidth]{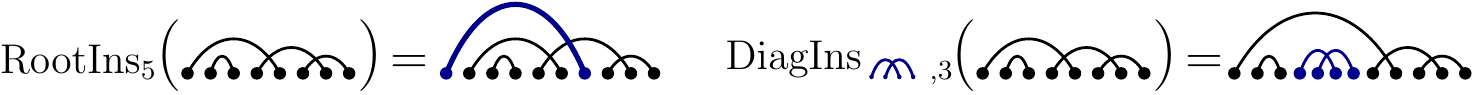}
\caption{Illustration of operations $\Rootdiag$ and $\Diagins$.}
\label{fig:diagraminsertion}
\end{figure}

The following technical lemma describes an important commutation relation between $\Rootdiag$ and $\Diagins$.

\begin{lemma}
Let $k$ and $\ell$ be two integers and $D$ an indecomposable chord diagram. We have the commutation rules
%\begin{equation}
%\Diagins_{C,k} \circ \Diagins_{D,\ell} = \Diagins_{D,\ell + 2 |C|} \circ \Diagins_{C,k},  
%\end{equation}
\begin{align}
 \DiagIns D {\ell} \circ \RootChord k &= \RootChord k \circ \DiagIns D {\ell-2},  &\textrm{ if }k \leq \ell - 2,   \label{eq:com1}\\ 
 \DiagIns D {\ell} \circ \RootChord k
  & =  \RootChord {k+ 2 |D|} \circ \DiagIns D {\ell-1},   & \textrm{ if } 1 \leq \ell - 1 \leq k, 
    \label{eq:com2}
\end{align}
where $|D|$ is the number of chords in $D$.
\label{lem:commut}
\end{lemma} 

\begin{proof} Each time we $(i)$ insert a new root chord into a diagram $C$ and then $(ii)$ insert a diagram into $C$, we can choose to do it in the opposite order -- $(ii)$ then $(i)$ -- as long as the diagram is not inserted into the first interval. The only things we have to take care of are the positions where the insertions occur, which can change after a root chord insertion or a diagram insertion. Thus, the $i$th leftmost interval becomes, after an operation $\RootChord k$, the $(i+1)$th leftmost interval if $i < k$, and the $(i+2)$th one if $i > k$. Similarly, after an operation $\DiagIns D \ell$, the $i$th leftmost interval remains the $i$th leftmost interval if $i < \ell$, and will become the $(i + 2|D|)$th leftmost interval if $i > \ell$.
Equations~\eqref{eq:com1} and~\eqref{eq:com2} follow from this analysis.
\end{proof}

Finally, we define a \textit{boxed product}\footnote{The terminology comes from Flajolet and Sedgewick~\cite[p. 139]{Flajolet-Sedgewick}. It means that we insert a combinatorial object into another at a particular place.	} for connected diagrams, which exactly corresponds to the combination of two connected diagrams described in the proof of Proposition~\ref{prop:recurrence}. 

\begin{definition}[Boxed product for connected diagrams]
Let $C_1$ and $C_2$ be two connected diagrams, and $i$ be an integer between $1$ and $2|C_2|-1$, where $|C_2|$ is the size of $C_2$. The connected diagram $C_1 \star_i C_2$ is defined as
\begin{align*}
\RootChord i (C_2) && \textrm{ if }C_1\textrm{ is the one-chord diagram,} \\
\RootChord {i+\ell} \left( \DiagIns {C_2} \ell (\widehat{C_1}) \right) && \textrm{ if }C_1\textrm{ is of the form }\RootChord \ell (\widehat{C_1}).
\end{align*}
\label{def:prod}
\end{definition}

Examples of this operation are shown in Figure~\ref{fig:diagramproduct}. Let us recall, as used in the proof of Proposition~\ref{prop:recurrence}, that the star product induces a bijection between connected diagrams $C$, and triples $(C_1,C_2,i)$ where $C_1$ and $C_2$ are two connected diagrams, and $i \, \in \, \{1,\dots,2|C_2|-1\}$.

Other similar definitions are both possible and useful.  We will define a variant of the boxed product for some technical work in 
Subsection~\ref{subsec bin tree} (see Definition~\ref{def var box}).

\begin{figure}[!ht]
\centering
\includegraphics[width = \textwidth]{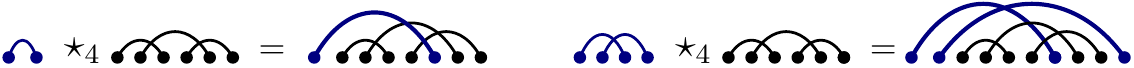}
\caption{Illustration of the boxed product for connected diagrams.}
\label{fig:diagramproduct}
\end{figure}

\subsection{The Bridge First Labeling of a map}
\label{ss:labeling}

% In other words, we have not defined an explicit bijection. To solve this problem, we need to know how the marking (counted by $2k-1$ in \eqref{prop:recurrence}) goes from one object to the other. A way to proceed is to label the intervals of a diagram with $n$ chords, and the corners of a map of size $n$, with the numbers $1,2,\dots,2n-1$. 

\begin{figure}[!ht]
\centering
\includegraphics[scale=1]{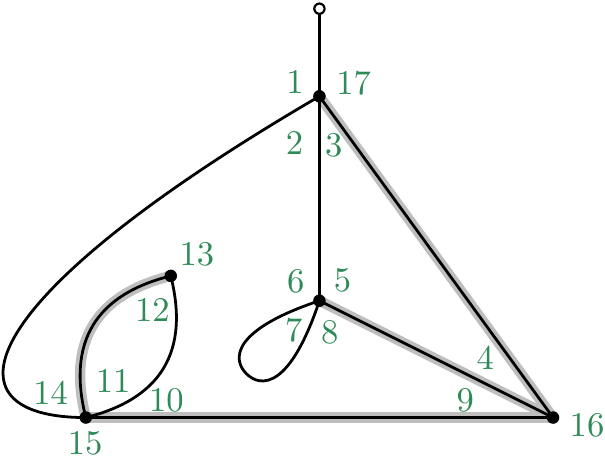}
\caption{The Bridge First Labeling of a map. The overlined edges correspond to the edges which become bridges during the algorithm described in Subsection \ref{ss:labeling}. Alternatively, they form the spanning tree associated to the  rightmost DFS.}\label{corner_numbering}
\end{figure}

Given a rooted map $M$ (potentially with dangling edges), we describe in this subsection a way to label the corners of $M$, which we call the \textit{Bridge First Labeling} of $M$. We choose this labeling because we want the operations of insertions in maps to satisfy an analogue of Lemma~\ref{lem:commut}.

The Bridge First Labeling is given by the following algorithm.
\begin{itemize}
\item The first corner we consider is the root corner. We label it by $1$.
\item Assume the current corner is labeled by $k$, and consider the (potentially dangling) edge $e$ adjacent to this corner in the counterclockwise order.  There are three possibilities:
\begin{itemize}
\item The edge $e$ is a bridge. Go along this edge to the next corner. 
%(without crossing $e$). 
Label this corner by $k+1$. 
\item The edge $e$ is a dangling edge. Go to the following corner in the counterclockwise order, and label it by $k+1$.
\item The edge $e$ is neither dangling nor a bridge. Cut $e$ into two dangling edges. Go to the following corner in the counterclockwise order, and label it by $k+1$. 
\end{itemize}
\item The algorithm stops when we reach the root.
\end{itemize}

An example of a run of this algorithm has been started in Figure~\ref{fig:bfl-algo}.

\begin{figure}[!ht]
\centering
\includegraphics[width = \textwidth]{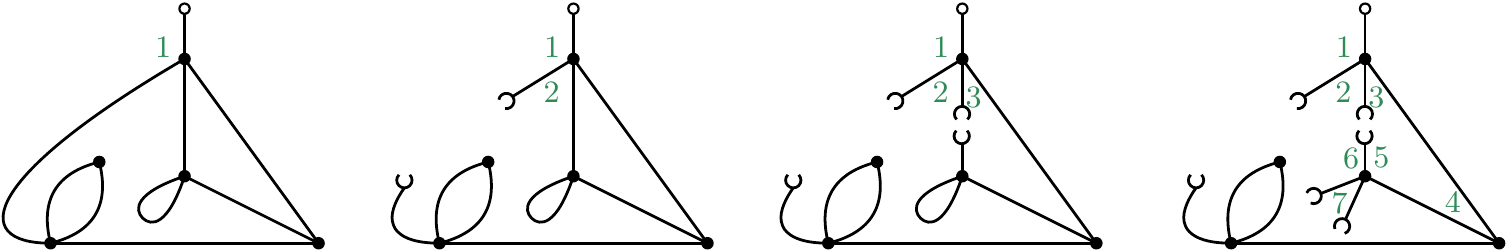}
\caption{The first steps of the Bridge First Labeling of the map of Figure~\ref{corner_numbering}.}
\label{fig:bfl-algo}
\end{figure}

Alternatively, the Bridge First Labeling can be deduced from the tour\footnote{in the sense of \cite{BernardiTutte}: we visit every half-edge, starting by the root. If a half-edge does not belong to the spanning tree, we go to the next half-edge in counterclockwise order; it a half-edge does belong to it, we follow the associated edge.} of the spanning tree induced by the  Depth First Search (DFS) of the map where we favor the rightmost edges (call this a \textit{rightmost DFS}). The notion of rightmost DFS will return in Subsection~\ref{subsec new dfs}.

\subsection{Operation in maps}

Now that we have set a suitable way to label the corners of a map, we define two analogues of $\Rootdiag$ and $\Diagins$ for maps:
\begin{definition}[Operations $\Rootmap$ and $\Mapins$]
Let $M$ be a map of size $n$, $k$ an integer $1 \le k \le 2n-1$, and $M'$ an arbitrary map.
We write $\RootChord k(M)$ to denote the map obtained from $M$ by adding an edge linking the root corner and the $k$th corner of the Bridge First Labeling of $M$.
We write $\MapIns {M'}k(M)$ to denote the insertion of $M'$ in $M$ via a bridge at the $k$th corner of the Bridge First Labeling of $M$.
\end{definition}
Examples are given by Figure \ref{fig:mapinsertion}.

\begin{figure}[!ht]
\centering
\includegraphics[width = \textwidth]{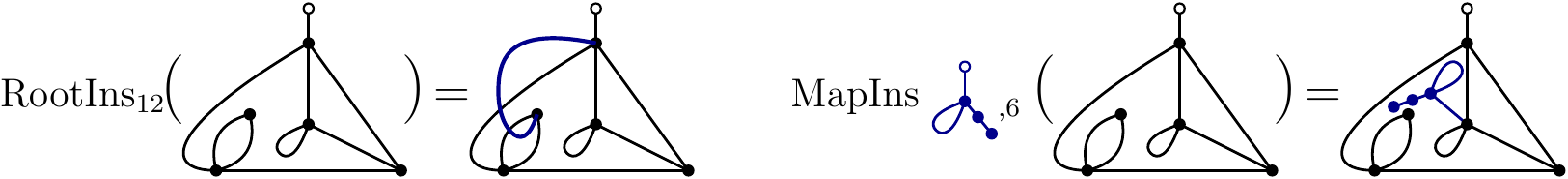}
\caption{Illustration of operations $\Rootmap$ and $\Mapins$.}
\label{fig:mapinsertion}
\end{figure}

The next lemma explains why we have chosen the Bridge First Labeling as a canonical way to number the corners of a map: the operations $\Rootmap$ and $\Mapins$ satisfy an analogous commutation relation as the corresponding operations $\Rootdiag$ and $\Diagins$ on diagrams (Lemma~\ref{lem:commut}). Numerous statistics will be thus preserved when we transform a map into a diagram.

\begin{lemma}
Let $k$ and $\ell$ be two integers, and $M$ be a combinatorial map (with only one dangling edge, marking the root). We have the commutation rules
\begin{align}
 \MapIns M {\ell} \circ \RootEdge k &= \RootEdge k \circ \MapIns M {\ell-2},  &\textrm{ if }k \leq \ell - 2,   \\ 
 \MapIns M {\ell} \circ \RootEdge k
  & =  \RootEdge {k+ 2 |M|} \circ \MapIns M {\ell-1},   & \textrm{ if } 1 \leq \ell - 1 \leq k, 
\end{align}
where $|M|$ is the number of edges in $M$.
\label{lem:commut2}
\end{lemma} 

\begin{proof}
Similarly as in Lemma~\ref{lem:commut}, we have to understand how a root edge insertion or a diagram insertion affects the labels of a map. 

The edge added by the operation $\RootEdge k$ will be necessarily cut in half at the start of the Bridge First Labeling algorithm. The rest of the tour will be unchanged, except for an extra step at the $k$th position, which is the visit of the second dangling edge resulting from the root edge. Therefore,
% the labels will be barely modified: 
a corner labeled by $i$ with $i < k$ will carry the label $i+1$ (the first dangling edge has been visited but not the second one),  
while a corner labeled by $i$ with $i < k$ will carry the label $i+2$.

Concerning the operation $\MapIns M \ell$, it will only affect the labels of the corners which are after $\ell$. Indeed, after the $\ell$th step, we have to visit the entire map $M$, which counts $2|M|$ corners. Thus, a corner with label $i > \ell$ will carry the label $i + 2|M|$ after the  operation $\MapIns M \ell$.
\end{proof}

Finally, we define a boxed product for bridgeless maps. As for connected diagrams, this product describes the combination between two bridgeless maps which is stated in the proof of Proposition~\ref{prop:recurrence}. It is the formal analog of Definition~\ref{def:prod}.

\begin{definition}[Boxed product for bridgeless maps] Let $M_1$ and $M_2$ be two bridgeless maps, and $i$ an integer  between $1$ and $2|M_2|-1$, where $|M_2|$ is the size of $M_2$. The bridgeless map $M_1 \star_i M_2$ is defined as
\begin{align*}
\RootEdge i (M_2) && \textrm{ if }M_1\textrm{ is reduced to a root,} \\
\RootEdge {i+\ell} \left( \MapIns {M_2} \ell (\widehat{M_1}) \right) && \textrm{ if }M_1\textrm{ is of the form }\RootEdge \ell (\widehat{M_1}).
\end{align*}
\label{def:prod2}
\end{definition}

Once again, following the proof of Proposition~\ref{prop:recurrence}, for each bridgeless map $M$ of size $>1$, there exists a unique triple $(M_1,M_2,i)$ where $M_1$ and $M_2$ are  two bridgeless maps such that $M = M_1 \star_i M_2$. Examples of this boxed product are shown in Figure~\ref{fig:mapproduct}.

\begin{figure}[!ht]
\centering
\includegraphics[width = 0.7\textwidth]{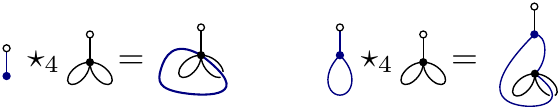}
\caption{Illustration of the boxed product for bridgeless maps.}
\label{fig:mapproduct}
\end{figure}

%***********************************%
%%%%%%%%%%%%%%%%%%%%%%%%%%%%%%%%%%%%
\section{Description of the main bijections}
\label{sec:bijection}
%%%%%%%%%%%%%%%%%%%%%%%%%%%%%%%%%%%%
%***********************************%

With all the tools we have introduced, it is now easy to construct explicit bijections between connected diagrams and bridgeless maps.

% We prove in this section Theorem~\ref{theo:simplebijection}.

%%%%%%%%%%%%%%%%%%%%%%%%%%%%%%%%%%%%
\subsection{Natural bijections}
%%%%%%%%%%%%%%%%%%%%%%%%%%%%%%%%%%%%

We establish first a bijection between bridgeless maps and connected diagrams, which we denote $\theta$. 
\begin{definition}[Bijection $\theta$ between bridgeless maps and connected diagrams]
Let $M$ be a bridgeless map.
\begin{itemize}
\item If $M$ is reduced to a root, then $\theta(M)$ is the one-chord diagram.
\item Otherwise, $M$ is of the form $M_1 \star_i M_2$. Then $\theta(M)$ is equal to  $\theta(M_1) \star_i \theta(M_2)$, where $\theta(M_1)$ and $\theta(M_2)$ are computed recursively.
\end{itemize}
\label{def:theta}
\end{definition}

\begin{figure}[!ht]
\centering
\includegraphics[scale=1.3]{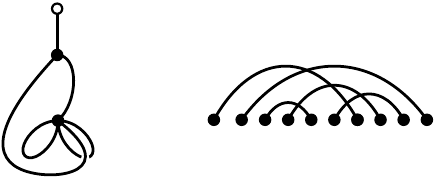}
\caption{A bridgeless map and a connected diagram in bijection under $\theta$.}
\label{fig:thetaex}
\end{figure}

The mapping $\theta$ is provably bijective since we can define its inverse $\theta^{-1}$ by symmetry.
Figure~\ref{fig:thetaex} presents a  bridgeless map and a connected diagram in bijection under $\theta$, the decompositions of which are shown by Figures~\ref{fig:diagraminsertion} and ~\ref{fig:mapinsertion}. 

%Table~\ref{tab:smallex} also shows the outputs of this bijection on small examples.

%There is one obstacle though: while one can obviously number the intervals of a diagram from left to right, there is no clear choice for the corners of a map. We describe one way to label the corners of a map in the next subsection which will provide the most significant results. 

%%%%%%%%%%%%%%%%%%%%%%%%%%%%%%%%%%%%%%%
%\subsection{Indecomposable diagrams and general maps}
%%%%%%%%%%%%%%%%%%%%%%%%%%%%%%%%%%%%%%%

As mentioned in the introduction, it was already known that rooted maps are in bijection with indecomposable diagrams \cite{OdMRtrans, OdMRencoding, Cori2009}. However, this known bijection does not restrict to a bijection between bridgeless maps and connected diagrams, so we will now give one which does.

\begin{definition}[Bijection $\phi$ between maps and indecomposable diagrams]\label{defn:phi}
Let $M$ be a combinatorial map. We define here the indecomposable diagram $\phi(M)$ as follows. (Figure~\ref{fig:phi2} illustrates this definition.)
\begin{itemize}
\item If $M$ is reduced to the root, then $\phi(M)$ is the one-chord diagram.
\item Assume that the root edge of $M$ is a bridge, i.e. $M$ is of the form $\MapIns {M_\downarrow} 1 (M_\uparrow)$. Then $\phi(M)$ is defined as 
\[\phi(M)= \DiagIns {\phi\left(M_\downarrow\right)} 1 \left(\phi(M_\uparrow)\right).\] (The diagrams $\phi(M_\downarrow)$ and $\phi(M_\uparrow)$ are defined recursively.)
\item Assume that the root edge of $M$ is not a bridge, i.e. $M$ is of the form $M = \RootEdge k (M')$. Then $\phi(M)$ is defined as
\[\phi(M)=\RootChord k \left(\phi(M')\right).\]
(The diagram $\phi(M')$ is defined recursively.)
\end{itemize}
\label{def:phi2}
\end{definition}

\begin{figure}[!ht]
\centering
\includegraphics[width = \textwidth]{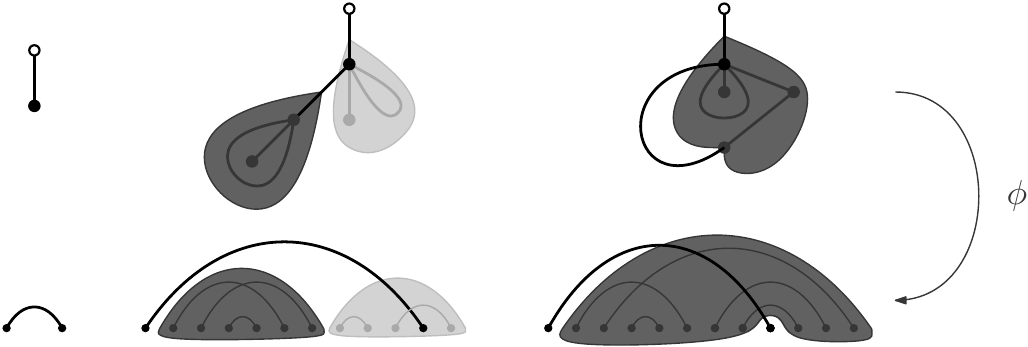}
\caption{How $\phi$ is defined. }\label{fig:phi2}
\end{figure}

Remarkably, the two previous  bijections are compatible with each other.
\begin{theorem} The bijection $\phi$ is a bijection between rooted maps and indecomposable diagrams whose restriction to bridgeless maps is $\theta$. (Therefore, $\phi$ sends bridgeless maps to connected diagrams.)
\label{theo:restriction}
\end{theorem}
The proof will be postponed for the next subsection.

\subsection{Extension of $\theta$ and equality between bijections}

In this subsection, we give another description of $\phi$, which is directly based on $\theta$.
To do so, we again exploit the fact that rooted maps and indecomposable diagrams have equivalent decompositions, but now in terms of bridgeless maps and connected diagrams. %In terms of generating functions, these decompositions relate the genera\-ting function $I(z)$ of indecomposable diagrams/general maps, and the generating function $C(z)$ of connected diagrams/bridgeless maps, through the (classical) formula
%\begin{equation}
%I(z) = (1 - I(z))\,C\left(z\,(1-I(z))^{-2} \right).
%\label{eq:Iz}
%\end{equation} 
The next proposition states those decompositions for both families, the principle of which is illustrated in Figure~\ref{fig:decomposition}.

\begin{proposition}\textbf{Decomposition of diagrams.} Any indecomposable diagram $D$ can be uniquely decomposed as a connected diagram $C$ and a sequence $(D_1,i_1),\dots,$ $(D_k,i_k)$  where each $D_j$ is an indecomposable diagram and $i_j$ is a integer such that $i_1 \leq \dots \leq i_k$ and
\[D = \DiagIns {D_1} {i_1} \circ \DiagIns {D_2} {i_2} \circ \dots \circ \DiagIns {D_k} {i_k} (C).\]
\textbf{Decomposition of maps.} Any map $M$ can be uniquely decomposed as a bridgeless map $M_B$ and a sequence $(M_1,i_1),\dots,(M_k,i_k)$  where each $M_j$ is a map and $i_j$ is a integer such that $i_1 \leq \dots \leq i_k$ and
\[M =  \MapIns {M_1} {i_1} \circ \MapIns {M_2} {i_2} \circ \dots \circ \MapIns {M_k} {i_k} (M_B).\]
\label{prop:decomposition}
\end{proposition}

\begin{figure}[!ht]
\centering
\includegraphics[width = \textwidth]{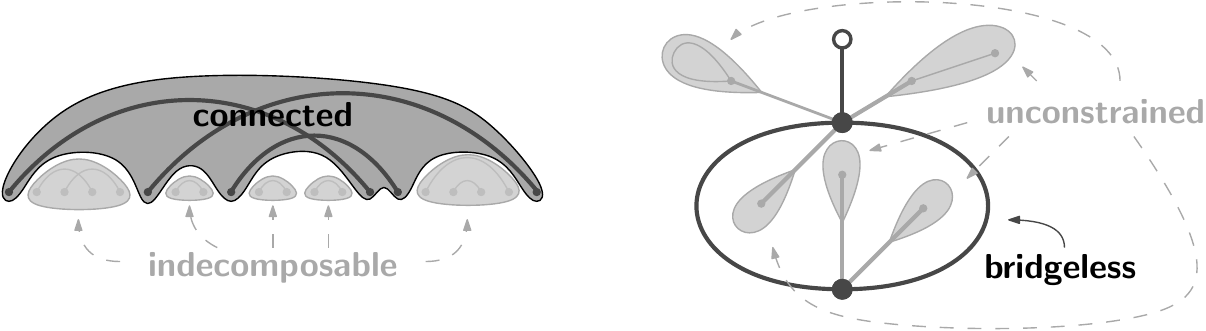}
\caption{\textit{Left.} Decomposition of an indecomposable diagram. \textit{Right.} Decomposition of an unconstrained map. It is also the image of the diagram under $\phi=\overline\theta$. }\label{fig:decomposition}
\end{figure}
%\nz{Typo in Fig.~\ref{fig:decomposition} (Images/decomposition.pdf): irreducible $\to$ indecomposable} % done

\begin{proof}
\noindent \textbf{Indecomposable diagrams.}
Here $C$ is the connected component of $D$ that includes the root chord. We can recover $D$ from $C$ by inserting in each interval of $C$ a sequence of indecomposable diagrams. We can do that starting from the right and ending to the left, which gives the above decomposition.

\noindent \textbf{Maps.} Here $M_B$ is the ``bridgeless component'' of the root (see right side of Figure~\ref{fig:decomposition}).
We recover $M$ from $M_B$ by grafting on each corner of $M_B$ a sequence of combinatorial maps. This can be done in the decreasing order for the Bridge First Labeling of $M_B$.
% \nz{Should also say something about how $M_B$ is computed from $M$ (analogous to how $C$ is computed from $D$ as the connected component of the root chord). Is there a standard name for this concept of the ``bridgeless core'' of a graph?}
% \jc{Do you think it could be problematic for the reader? I agree that the "inverse" part is missing, but I thought it was clear enough (especially with the figure). If the notion of bridgeless core exists, I'm unaware of it. }
% \nz{I went ahead and added a couple words, calling this the ``bridgeless component'' rather than ``bridgeless core'' to harmonize with your terminology in the conclusion.}
\end{proof}

\newcommand{\otheta}{\overline \theta}

\begin{definition}[Definition of $\otheta$]
Consider a map $M$. Let \[M =  \MapIns {M_1} {i_1} \circ \MapIns {M_2} {i_2} \circ \dots \circ \MapIns {M_k} {i_k} (M_B)\] be the decomposition of $M$ described by Proposition~\ref{prop:decomposition}. Then $\otheta(M)$ is defined as the diagram 
\[\otheta(M) = \DiagIns {\otheta(M_{1})} {i_1} \circ \DiagIns {\otheta(M_{2})} {i_2} \circ \dots \circ \DiagIns {\otheta(M_{k})} {i_k} \left( \theta(M_B) \right).\]
where $\theta$ is the bijection defined by Definition~\ref{def:theta} and where $\otheta(M_{1}),\dots,\otheta(M_{k})$  are computed recursively\footnote{Since we have $\otheta = \theta$ for bridgeless diagrams, the base cases of the recursion are well treated.}.
\end{definition}

It is easy to prove that $\otheta$ is a bijection since $\otheta^{-1}$ can be similarly defined by swapping the roles of maps and diagrams. Moreover, when $M$ is bridgeless, we have $k=0$. Therefore, the restriction of $\otheta$ to bridgeless maps is, by definition, equal to $\theta$.

Theorem~\ref{theo:restriction} then results from the following proposition.

\begin{proposition} We have $\phi=\otheta$. % A sentence cannot begin by a mathematical symbol.
\label{prop:phi=theta}
\end{proposition}

\begin{proof} We prove that $\phi(M) = \otheta(M)$ for any map $M$ by induction on the size of the map. The base case (when $M$ reduced to a root) is given by the definitions.

Let $M$ be a map of size $>1$, which we decompose (by Proposition~\ref{prop:decomposition}) as \[M =  \MapIns {M_1} {i_1} \circ \MapIns {M_2} {i_2} \circ \dots \circ \MapIns {M_k} {i_k} (M_B).\] There are three possibilities.

\textbf{1. The root edge of $M$ is a bridge.}  Since the root edge of $M$ is a bridge, we have $i_1 = 1$. Moreover, referring to the notation of Definition~\ref{def:phi2}, $M_1=M_\downarrow$ and $M_\uparrow =  \MapIns {M_2} {i_2} \circ \dots \circ \MapIns {M_k} {i_k} (M_B)$. By using twice the definition of $\otheta$, we have 
\begin{align*}
\otheta(M) =& \DiagIns {\otheta(M_{1})} 1 \circ \DiagIns {\otheta(M_{2})} {i_2} \circ \dots \circ \DiagIns {\otheta(M_{k})} {i_k} \left( \theta(M_B) \right) \\ =& \DiagIns {\otheta(M_{1})} 1 \left( \otheta(M_\uparrow)\right).
\end{align*}
But by induction, $\otheta(M_1)=\phi(M_1)$ and $\otheta(M_\uparrow) = \phi(M_\uparrow)$. Thus, we recover the definition of $\phi$, and so $\otheta(M) = \phi(M)$.

\textbf{2. The root edge of $M$ is not a bridge and its deletion in $M_B$ gives a bridgeless map $M'_B$.} Then $M_B$ is of the form $M_B = \RootEdge i (M'_B)$.  By definition of $\theta$, we have $\theta(M_B) = \RootChord i (\theta(M'_B))$.
Therefore
\[
\otheta(M) = \DiagIns {\otheta(M_{1})} {i_1} \circ \dots \circ \DiagIns {\otheta(M_{k})} {i_k} \circ \RootChord i\left( \theta(M'_B) \right).\]
Since $i_1 > 1$, we can use Lemma~\ref{lem:commut2} to slide the operation $\Rootdiag$ to the left:
\[
\otheta(M) = \RootChord j \circ \DiagIns {\otheta(M_{1})} {j_1} \circ \dots \circ \DiagIns {\otheta(M_{k})} {j_k} \left( \theta(M'_B) \right).\]
(The integers $j$, $j_1, \dots, j_k$ are given by Lemma~\ref{lem:commut2}.) But by Lemma~\ref{lem:commut} %(the analog of Lemma~\ref{lem:commut2} for maps)
we also have 
\begin{align*}
M =  & \MapIns {M_1} {i_1}  \circ \dots \circ \MapIns {M_k} {i_k} \circ \RootEdge i (M'_B) \\
= & \RootEdge j  \circ \MapIns {M_1} {j_1} \circ \dots \circ \MapIns {M_k} {j_k} (M'_B),
\end{align*}
with the same $j,j_1,\dots,j_k$ as above. So using successively the definition of $\phi$, the induction hypothesis, and the definition of $\overline\theta$,
\begin{align*}
\phi(M) =  & \RootEdge j   \left(\phi \left(\MapIns {M_1} {j_1} \circ \dots \circ \MapIns {M_k} {j_k} (M'_B) \right)\right) \\
= & \RootEdge j  \left( \otheta \left(\MapIns {M_1} {j_1} \circ \dots \circ \MapIns {M_k} {j_k} (M'_B) \right)\right) \\
= & \RootEdge j \left( \DiagIns {\otheta(M_{1})} {j_1} \circ \dots \circ \DiagIns {\otheta(M_{k})} {j_k} \left( \theta(M'_B) \right)\right) = \overline\theta(M).
\end{align*}
%Comparing this with the above expression of $\otheta(M)$. We have the equality $\phi(M)=\otheta(M)$.

\textbf{3. The root edge of $M$ is not a bridge and its deletion in $M_B$ does not give a bridgeless map.}  Since the deletion of the root edge of $M_B$ does not give a bridgeless map, $M_B$ is a boxed product (see Definition~\ref{def:prod2}) of the form \[M_B = M' \star_\ell M''\] where $M'$ is a bridgeless map of size $>1$, and $M''$ is some bridgeless map. Since $M'$ has size more than $1$, it can be put in the form $\RootEdge i (\widehat M)$. Then, by definition of the boxed product, $M_B$ can be written as 
\[ M_B = \RootEdge {i+\ell} \circ \MapIns {M''} {i} {\left(\widehat M\right)}.\]
Since by definition, $\theta(M_B) = \theta(M') \star_\ell \theta(M'')$, we also  have 
\[ \theta(M_B) = \RootChord {i+\ell} \circ \DiagIns {\theta(M'')} {i} \left(\theta(\widehat M)\right).\]
%So using successively $\phi$ until obtaining a map $\hat M$ in which the deletion of the root edge gives a bridgeless map, one can prove that 
%\[ \phi(M_B) = \RootEdge {i+ \ell_1 + \dots + \ell_m} \circ \DiagIns {\phi(M'_1)} {i+ \ell_2 + \dots + \ell_m} \circ \dots \circ \DiagIns {\phi(M'_m)} {i} \left({\phi(\widehat M)}\right),\]
%where
%\[ M_B = \RootChord {i+ \ell_1 + \dots + \ell_m} \circ \MapIns {M'_1} {i+ \ell_2 + \dots + \ell_m} \circ \dots \circ \MapIns {M'_m} {i} {(\widehat M)}.\]
Then, using the same techniques as the previous case, we apply the definition of $\otheta$:
\[ \small
\otheta(M) = \DiagIns {\otheta(M_{1})} {i_1} \circ \dots \circ \DiagIns {\otheta(M_{k})} {i_k} \circ \RootChord {i+\ell} \circ \DiagIns {\theta(M'')} {i} \left(\theta(\widehat M)\right),\]
we commute the operators thanks to Lemma~\ref{lem:commut}:
\[
\otheta(M) = \RootChord j \circ \DiagIns {\otheta(M_{1})} {j_1} \circ \dots \circ \DiagIns {\otheta(M_{k})} {j_k}  \circ \DiagIns {\theta(M'')} {i} \left(\theta(\widehat M)\right),\]
we recognize the definition of $\otheta$:
\[
\otheta(M) =  \RootEdge j  \left( \otheta \left(\MapIns {M_1} {j_1} \circ \dots \circ \MapIns {M_k} {j_k} \circ \MapIns {M''} {i} \left(\widehat M \right) \right) \right), \]
and we use the induction hypothesis and the definition of $\phi$ to conclude.
\end{proof}

%\ky{there was a remark here which I commented out for brevity}
%An even more convincing description of this bijection can be given in terms of root edges insertion in the style of Tutte (see the article of Arquès and Béraud \cite{ABmaps} for the decomposition of rooted maps -- the one for indecomposable diagrams can be found similarly). 
%\nz{I'm not exactly sure what you mean. We've emailed before about the recurrence $b_{n+1} = (2n-1)b_n + \sum_{k=1}^n b_k b_{n+1-k}$ which is listed at A000698, and which can be justified by a Tutte style recursion asking whether or not the root is a bridge. Is that what you are referring to? That's isn't in \cite{ABmaps}, though.}

\subsection{Planar maps as diagrams with forbidden patterns}\label{sec forbidden}

%is the restriction of planarity on the map side.  

%Looking through some examples it is easy to notice that non-planar maps correspond to chord diagrams which have extra crossings in some sense. 
%
% The precise structure which is forbidden in chord diagrams corresponding to planar maps is as illustrated in Figure~\ref{fig forbidden}.  Observe that any 3-crossing is an example of this forbidden configuration, but that some diagrams with no 3-crossings do have the forbidden configuration. \jc{not defined}

Planarity of a combinatorial map can be recognized using its Euler characteristic.
\begin{definition}[Faces, Euler characteristic, planarity]
  Let $M = (H,\sigma,\alpha,r)$ be a rooted combinatorial map (potentially with dangling edges).
  The \definand{faces} of $M$ are the orbits of the composite permutation $\sigma\alpha$.
  The \definand{root face} is the face containing $r$.
  The \definand{Euler characteristic} of $M$ is defined by
\[ \chi(M) = |\orbit(\sigma)| + |\orbit(\alpha)| + |\orbit(\sigma\alpha)| - |H|.
\]
$M$ is said to be \definand{planar} if $\chi(M) = 2$.
\end{definition}

We here characterize the image of planar maps under the previous bijections.  

\begin{figure}[!ht]
\centering
\includegraphics{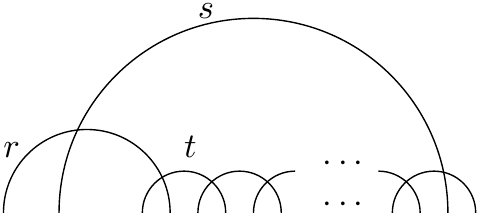}
\caption{Forbidden configuration for diagrams corresponding to planar maps.}\label{fig forbidden}
\end{figure}

\begin{proposition}
\label{prop:bij}
Under $\phi$ planar rooted maps with $n$ edges are in bijection with indecomposable diagrams with $n$ chords which do not contain the configuration of Figure~\ref{fig forbidden} as a subdiagram.
%which do not have two crossing chords $c_1$ and $c_2$ with a chain of crossing chords linking the larger end points of $c_1$ and $c_2$ and with the chain having no further crossings of $c_1$ or $c_2$, see Figure~\ref{fig forbidden}.  
Thus, restricting to $\theta$, a bridgeless map is planar if and only if the corresponding connected diagram does not contain the forbidden configuration.
\end{proposition}

Before we prove this result we need a couple more definitions.

\begin{definition}[Internal/external corners]
  Given a planar rooted map $M$, a corner whose second component is contained in the root face is called an \definand{external corner} of $M$.
  A corner which is not external is called \definand{internal}.
\end{definition}

\begin{definition}[Blocked/unblocked intervals]
  Given an indecomposable diagram $D$, an interval in $D$ is a \definand{blocked interval} if it is 
  \begin{itemize}
  \item under the root chord and under at least one other chord in the same connected component as the root chord,
    
  \item or already blocked in a component of the diagram obtained by removing the root chord.
  \end{itemize}
  An interval which is not blocked is called \definand{unblocked}.
\end{definition}

\begin{lemma}\label{lem blocked}
  Let $M$ be a planar rooted map.  The blocked intervals of $\phi(M)$ correspond to the internal corners of $M$.
\end{lemma}

\begin{proof}
We prove this lemma by induction.  In the base case $M$ is the trivial map, which has no internal corners, corresponding to the one-chord diagram with no blocked intervals.  Suppose $M$ is a planar map with more than one half-edge.  There are two cases.

Suppose the root edge of $M$ is a bridge, so that $M=\MapIns {M_1}1(M_2)$, where $M_1$ and $M_2$ are planar maps.
All of the internal corners of $M_1$ and $M_2$ remain internal in $M$, and all of the external corners remain external (the external root corner of $M_2$ splits into two external corners in $M$).
Likewise, since the connected components remain the same, all of the blocked intervals of $\phi(M_1)$ and $\phi(M_2)$ remain blocked in $\phi(M) = \DiagIns {\phi(M_1)} 1 (\phi(M_2))$, and unblocked intervals remain unblocked.
By induction, the internal corners of $M_1$ and $M_2$ correspond to the blocked corners of $\phi(M_1)$ and $\phi(M_2)$, so this ends the proof.

%$\phi(M) = \DiagIns {\phi(M_1)} 1 (\phi(M_2))$ which inserts the diagram $\phi(M_1)$ under the root of $\phi(M_2)$ but does not change the connected components, so the blocked intervals of $\phi(M)$ built by the first point are exactly those so built in $\phi(M_1)$ and the blocked intervals of $\phi(M)$ built by the second point are all blocked intervals of $\phi(M_2)$ along with those built by the second point in $\phi(M_1)$.  \nz{What do you mean by ``point''?} So the blocked intervals are unchanged which is what we wanted.

The other case is $M=\RootEdge k(M_1)$, where $M_1$ is planar and $k$ is an external corner of $M_1$.
The external corners of $M$ are its root corner, along with the external corners of $M_1$ which counterclockwisely follow the corner labeled by $k$. 
Expressed in terms of the Bridge First Labeling of $M_1$, these external corners are those with an index larger than $k$, which correspond to the corners with index $\ge k+2$ in $M$.
On the other hand, for the diagram $\phi(M) = \RootChord k(\phi(M_1))$, the new root chord blocks the intervals $2$ through $k+1$ in $\phi(M)$, while leaving the other intervals unchanged.
By induction, internal corners of $M_1$ correspond to blocked corners of $\phi(M_1)$, so this concludes the proof.
\end{proof}

\begin{proof}[Proof of Proposition~\ref{prop:bij}]
Let us first observe that $\MapIns {M_1} k (M_2)$ is planar if and only if both $M_1$ and $M_2$ are planar, while $\RootChord k (M_1)$ is planar if and only if $M_1$ is planar and $k$ is an external corner of $M_1$.

%  The first thing to resolve is under which circumstances $\Rootmap$ and $\Mapins$ yield non-planar maps when applied to planar arguments.
%If $M_1$ and $M_2$ are planar maps then $\MapIns {M_1} k (M_2)$ is always planar by embedding $M_1$ in the face of $M_2$ containing the corner $k$.  
%If $M$ is a planar map then $\RootChord k (M)$ is planar iff the corner $k$ is an external corner.

Now, consider a map $M$ as built iteratively according to the induction used in the definition of $\phi(M)$.

Suppose $M$ is nonplanar. Then at some stage in this construction we must have built a map $\RootEdge k (M_1)$ by inserting a new root edge into an internal corner of a planar map $M_1$.  Let $M'=\RootEdge k(M_1)$.  We will now proceed to show that $\phi(M')$ has the configuration of Figure~\ref{fig forbidden}.

By Lemma~\ref{lem blocked}, since $\RootEdge k (M_1)$ comes from the insertion of a new root edge into an internal corner, $\RootChord k (\phi(M_1))$ comes from the insertion of a a new root into a blocked interval $k$ of $\phi(M_1)$.  Let $r$ be the root chord of $\phi(M') = \RootChord k (\phi(M_1))$.
Since the interval where $r$ is inserted is blocked, there is some subdiagram of $\phi(M_1)$ where the first point of the definition of blocked interval holds. In other words, there is a connected subdiagram $C$ of $\phi(M_1)$ with root chord $s$, and when $r$ is inserted via $\RootChord k (\phi(M_1))$, then $r$ crosses both $s$ and another chord $t$ of $C$.  Since $s$ is the root of $C$, other chords of $C$ can only cross $s$ on the right.  Also $C$ is connected, so there is a chain of chords connecting $t$ to the right hand side of $s$.   By taking a minimum chain we can guarantee that the chords in the chain go from left to right and  do not cross chords which are not their immediate neighbors.   Thus $r$, $s$ and the chain give the forbidden configuration in Figure~\ref{fig forbidden}.

Further operations of $\Rootdiag $ and $\Diagins$ preserve the forbidden configuration, and so $\phi(M)$ also has the forbidden configuration.  Thus we have proved that if $M$ is nonplanar then $\phi(M)$ has the forbidden configuration.

Now consider the converse.  With no planarity assumption on $M$, suppose that $\phi(M)$ has the forbidden configuration.  Then at some stage in the construction of $\phi(M)$ we must have built a diagram $\RootChord k (\phi(M_1))$ so that the newly constructed root chord, call it $r$, crosses the root chord $s$ of $\phi(M_1)$ and another chord $t$ of $\phi(M_1)$ and there is a chain in $\phi(M_1)$ joining the right end points of $t$ and $s$.  In particular the $k$th interval  of $\phi(M_1)$ is under $s$ and $t$, both of which are in the same connected component of $\phi(M_1)$.  Thus this interval is blocked in $\phi(M_1)$.  By Lemma~\ref{lem blocked}, it must correspond to an internal corner of $M_1$, and so $\RootEdge k (M_1)$ is nonplanar.  Further, once the map becomes nonplanar, no sequence of $\Rootmap$ or $\Mapins$ operations can make the map planar again, and so $M$ is also nonplanar. 
% We have shown that if $\phi(M)$ has the forbidden configuration, then $M$ is nonplanar. This completes the proof.
\end{proof}

%%%%%%%%%%%%%%%%%%%%%%%%%%%%%%%%%%%%%%%
%************************************%
\section{New perspectives on chord diagram expansions in quantum field theory}
\label{sec:qft}
%************************************%
%%%%%%%%%%%%%%%%%%%%%%%%%%%%%%%%%%%%%%%

%%%%%%%%%%%%%%%%%%%%%%%%%%%%%%%%%
\subsection{Context}
\label{ss:qftcontext}
%%%%%%%%%%%%%%%%%%%%%%%%%%%%%%%%%

Interestingly, by the work of some of the authors with other collaborators~\cite{MYchord, HYchord, CYchord}, rooted connected chord diagrams appear in quantum field theory where they give series solutions to certain Dyson-Schwin\-ger equations. We are going to see that the $\theta$ bijection of Section~\ref{sec:bijection} will simplify some formulas in this theory: Corollary~\ref{cor:allinmaps} recasts the main result of \cite{HYchord} in map language; Table~\ref{tab:transfer} shows how important parameters translate, some becoming considerably more natural; and along the way we prove and generalize a conjecture of Hihn (see the discussion at the end of Subsection~\ref{subsec:new qft}).  

Dyson-Schwinger equations are the quantum analogues of the classical equations of motion. The solutions of such functional equations 
%(generally expressed as an integral)
are some Green functions of the quantum field theory. 
It turns out that these equations have a nice underlying combinatorial aspect.  They capture the decomposition of Feynman diagrams into subdiagrams, so viewing perturbative expansions as intricately weighted generating functions, the Dyson-Schwinger equations can be interpreted as equations for the generating functions of appropriate combinatorial classes of Feynman diagrams.  Furthermore these functional equations mirror the combinatorial decomposition of the graphs.  Using the universal property of the Connes-Kreimer Hopf algebra of rooted trees, we can also view the Dyson-Schwinger equations as functional equations for classes of rooted trees.  This happens by using the rooted trees to represent insertion structures of Feynman diagrams.
%
%One of us has a program with various collaborators to better understand the combinatorial underpinnings of Dyson-Schwinger equations.  One of the successes of this program has been to solve certain classes of Dyson-Schwinger equations using expansions indexed by chord diagrams, first in \cite{MYchord} and then generalized in \cite{HYchord}. % In \cite{MYchord} one of us with Nicolas Marie considered the Dyson-Schwinger equation corresponding to the situation where one primitive Feynman graph is inserted into itself in all possible ways on one internal edge.  After converting into differential form following \cite{Ymem} or \cite{kythesis} this gives the equation
%\[
%G(x,L) = 1 - xG(x, \partial_{-\rho})^{-1}(e^{-L\rho} - 1)F(\rho)|_{\rho=0}
%\] 
%This should be interpreted as an equation in formal series.  
%\[
%F(\rho) = f_0\rho^{-1} + f_1 + f_2\rho + \cdots
%\]
%is the series expansion of the regularized Feynman integral of the primitive Feynman graph.  See \cite{MYchord} for details.

One of us has a program with various collaborators to better understand the combinatorial underpinnings of Dyson-Schwinger equations.  One of the successes of this program has been to solve certain classes of Dyson-Schwinger equations using expansions indexed by chord diagrams, first in \cite{MYchord} and then generalized in \cite{HYchord}. 

The paper \cite{MYchord} considered the Dyson-Schwinger equation corresponding to the situation where one primitive Feynman graph is inserted into itself in all possible ways on one internal edge.  One important instance of a Dyson-Schwinger equation of this form was solved by Broadhurst and Kreimer in \cite{BKerfc}.  This  also gives a good concrete example of this kind of Dyson-Schwinger equation.  Diagrammatically this Dyson-Schwinger equation is
\begin{equation}\label{eq diagrammatic DSE}
\includegraphics{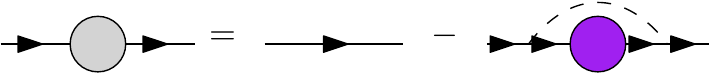}
\end{equation}
where the purple (darker) blob represents a sequence of any number of copies of the second term on the right hand side of the light grey blob equation, namely
\[
\includegraphics{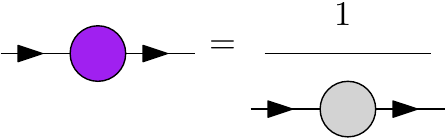}
\]
which is conventionally written as an inverse, meaning to expand as a geometric series.  The actual physical Dyson-Schwinger equation is an integral equation which arises when each Feynman graph is replaced by its Feynman integral, in this case with the Feynman diagrams viewed in massless Yukawa theory.  The diagrammatics are one notation for this equation.  An example of a diagram appearing in this expansion is given in Figure~\ref{fig Yukawa}.  Such Feynman graphs have a rooted tree structure by how they were constructed.

\begin{figure}\centering
  \includegraphics{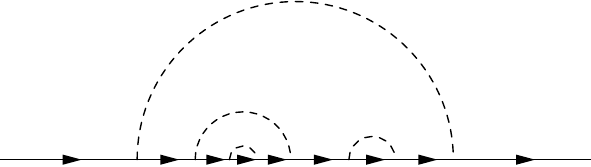}
  \caption{An example of a Feynman graph which occurs in the simpler Dyson-Schwinger equation.}\label{fig Yukawa}
\end{figure}

Returning to the set up of \cite{MYchord} and the connection to chord diagrams,
after being converted into differential form following \cite{Ymem} or \cite{kythesis},  the Dyson-Schwinger equation considered in \cite{MYchord} becomes
\begin{equation}\label{eq s 2 case}
G(x,L) = 1 - xG(x, \partial_{-\rho})^{-1}(e^{-L\rho} - 1)F(\rho)|_{\rho=0}
\end{equation}
where 
\[
F(\rho) = f_0\rho^{-1} + f_1 + f_2\rho + \cdots
\]
is the series expansion of the regularized Feynman integral of the primitive Feynman graph (see \cite{MYchord} for details).  Equation~\eqref{eq s 2 case} should be interpreted as a formal series equation.  

To continue the Broadhurst Kreimer example, in that case we have $F(\rho) = 1/(\rho(1-\rho))$ and so then $G(x,L)$ in~\eqref{eq s 2 case} is the sum indexed by all Feynman diagrams generated by~\eqref{eq diagrammatic DSE} where each diagram contributes its Feynman integral. The variable $L$ is defined as $L=\log(q^2/\mu^2)$ where $q$ is the momentum coming in and going out of each diagram and $\mu$ is a renormalization constant, and $x$ is the coupling constant (giving the strength of the interaction).

\medskip

To proceed, we need two further definitions
concerning rooted connected chord diagrams which arise from the quantum field theory application, see \cite{MYchord, HYchord}.  

\begin{definition}[Intersection order] The \definand{intersection order} of the chords of a rooted connected diagram $C$ is defined as follows.
\begin{itemize}
  \item The root chord of $C$ is the first chord in the intersection order.
  \item Remove the root chord of $C$ and let $C_1, C_2, \ldots, C_k$ be the connected components of the result ordered by their first vertex.
  \item For the intersection order of $C$, after the root chord come all the chords of $C_1$ ordered recursively in the intersection order, then all the chords of $C_2$ ordered by intersection order, and so on.
\end{itemize}
\end{definition}

  This intersection order is not in general the same as the order by first endpoint, see Figure~\ref{fig diff order} for an example. 
The intersection order and the order by first endpoint both define total orders on chords extending the partial order induced by paths in the intersection graph (recall Definition~\ref{def:intgraph-condiag}).

\begin{definition}[Terminal chord] A chord $c$ is \definand{terminal} if the left endpoint of every chord intersecting $c$ is to the left of $c$.
\end{definition}
Equivalently, a chord $c$ is terminal if it does not cross any chords larger than it in the intersection order; or (third equivalent definition) a chord is terminal if it is a sink in the intersection graph. For example, in Figure~\ref{fig diff order}, only chords labeled by $3$ and $4$ are terminal.

\begin{figure} \centering
  \includegraphics{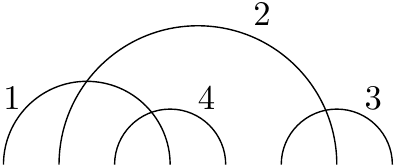}
  \caption{An example where the intersection order (indicated) is not the order by first end point.}\label{fig diff order}
\end{figure}

The main result of \cite{MYchord} was to solve the Dyson-Schwinger equation \eqref{eq s 2 case} as 
\begin{equation}\label{eq first dse sol}
G(x,L) = 1 - \sum_{C} \left(\sum_{i = 1}^{b(C)} f_{b(C)-i} \frac{(-L)^i}{i!} \right) x^{|C|}f_0^{|C|-\ell}\prod_{j=2}^{\ell}f_{t_j-t_{j-1}}
\end{equation}
where the sum is over connected diagrams $C$ and the terminal chords of $C$ are indexed by $b(C)=t_1 < t_2 < \cdots < t_{\ell}$ in the intersection order.  Note that this gives $G(x,L) - 1$ as a kind of strangely weighted generating function of connected diagrams. Its first terms are given by
\[ G(x,L) - 1 = f_0 L x + \left(f_1 L - f_0 \frac{L^2} 2 \right) f_0 x^2 + \dots ,\]
which respectively correspond to the one-chord diagram and the connected two-chords diagram.
In \cite{CYchord} two of us used tools of asymptotic combinatorics to better understand some of these parameters and in particular were able to conclude that in each of the next-to${}^{k}$-leading log expansions only $f_0$ and $f_1$ contribute.

We can compare~\eqref{eq first dse sol} to the original Feynman diagram expansion.  Both are expansions over combinatorial objects yielding the same series $G(x,L)$.  In the Feynman diagram expansion each diagram have a very complicated contribution, namely its Feynman integral, to the sum. Thus if we want to find properties like the asymptotic behavior of $G(x,L)$, the Feynman diagram expansion hides important features in the Feynman integrals and so only a combinatorial analysis can get us so far.  

  In the chord diagram expansion each chord diagram has a simple contribution to the sum -- just certain monomials in the $f_i$.  This means that, in principle, combinatorial tools could fully understand $G(x,L)$, and in practice we can make good progress as in \cite{CYchord}.  On the other hand, we have lost a physical interpretation for each diagram (the Feynman diagrams directly represent particles and their interactions); each chord diagram just represents some terms in expansions of some Feynman diagrams.

\medskip

In \cite{HYchord}, generalizing \cite{MYchord}, one of us with Markus Hihn solves the Dyson-Schwinger equation
%The main result of \cite{HYchord}, generalizing \cite{MYchord}, consists of solving a family of Dyson-Schwin\-ger equations in terms of a sum over weighted connected chord diagrams:
%
%\begin{theorem}[Hihn, Yeats~\cite{HYchord}] For a positive integer $s$, we define the Dyson-Schwinger equation
\begin{equation}\label{eq gen case}
G(x, L) = 1 - \sum_{k \geq 1}x^kG(x, \partial_{-\rho})^{1-sk}(e^{-L\rho}-1)F_k(\rho)_{\rho=0}
\end{equation}
where 
$F_k(\rho) = \sum_{i\geq 0} a_{k, i} \rho^{i-1}$ and $s$ is a positive integer parameter. 
This Dyson-Schwinger equation corresponds to the case where we are still restricted to propagator corrections but now we can have any number of primitive Feynman graphs (the integer $k$ refers to their possible sizes, where the size is the dimension of the cycle space of the graph), and the number of insertion places is one less than $s$ times the size of the graph, where $s$ can be any positive integer.  

The main result of \cite{HYchord} is that \eqref{eq gen case} is solved by  
\begin{equation}
G(x,L) = 1 - \sum_C \left(\sum_{i = 1}^{b(C)} a_{d(b(C)), b(C)-i} \frac{(-L)^i}{i!} \right) w(C)
%\prod_{\substack{c \text{ not}\\ \text{terminal}}}a_{d(c), 0}\prod_{j=2}^{\ell}a_{d(t_j), t_j-t_{j-1}}, 
A(C) x^{\|C\|},
\label{eq:solutionDS}
\end{equation}
%In this case we must sum 
where the first sum runs over all connected diagrams $C$, carrying a positive integer weight $d(c)$ on each of its chords $c$, and such that the position of the first terminal chord is $b(C)$. As for the other parameters, $|C|$ denotes the number of chords; $\|C\|$ is the sum of the chord weights; $t_1=b(C)<t_2<\dots<t_\ell = |C|$ lists the positions of all the terminal chords in intersection order; 
\begin{equation}
w(C) = \prod_{m=1}^{|C|}\binom{d(m)s + \nu(m) -2}{\nu(m)};
\label{eq:defw}
\end{equation}
and  
\begin{equation} A(C) = \prod_{\substack{c \text{ not}\\ \text{terminal}}}a_{d(c), 0}\prod_{j=2}^{\ell}a_{d(t_j), t_j-t_{j-1}}.
\label{eq:defA}
\end{equation}
For the definition of $w(C)$, we need another parameter $\nu(c)$ which is discussed in the next subsection.  Note that again $G(x,L) - 1$ is a weighted generating function of connected diagrams.

%\label{theo:HihnYeats}
%\end{theorem}
%\qquad
%\text{ and }
%\qquad
%A(C) = \prod_{\substack{c \text{ not}\\ \text{terminal}}}a_{d(c), 0}\prod_{j=2}^{\ell}a_{d(t_j), t_j-t_{j-1}}
%\]

  \begin{example}\label{eg chord params}
  As an example, take the diagram in Figure~\ref{fig diff order}.  Note that the terminal chords are chords $3$ and $4$, so $b(C)=3$.  If all the chords are weighted by $1$ then $A(C) = a_{1,0}^2a_{1, 1}$.  If the first chord is weighted by $2$ while the rest are weighted by $1$ then $A(C) = a_{1,0}a_{2,0}a_{1,1}$, while if the fourth chord is weighted by $2$ and the rest are weighted by $1$ then $A(C)= a_{1,0}^2a_{2,1}$.  Note that the weight of the first terminal chord does not affect $A(C)$.

  Continuing the example, note that if all chords are weighted by $2/s$ (since $s$ is an integer and the weights are nonnegative integers, this means $s=2$ and all weights are $1$ or $s=1$ and all weights are 2) then $w(C)$ is independent of $\nu$ and equals $1$ for all $C$.  More generally, $\nu(C)$ will be defined in the next subsection, but for now taking it as given that $\nu(1) = \nu(2) = 0$ and $\nu(3)=2$ and $\nu(4)=1$ then we can compute $w(C)$.  Say $s=2$ and all chords are weighted by $1$ except the third which has weight $2$, then $w(C) = \binom{3}{2} = 3$.  With the same weights but $s=3$ we get $w(C) = \binom{1}{0}\binom{1}{0}\binom{3}{2}\binom{7}{6} = 21$.
\end{example}
  
The theorem stating that \eqref{eq:solutionDS} solves \eqref{eq gen case} was shown by checking that the coefficients of the Dyson-Schwin\-ger equation and the eventual solution both satisfy the same recurrences with the same initial conditions. 
%The recurrence coming from a variant of the boxed product, defined in the next subsection, gives the relation between the coefficients of different powers of $L$.
This was done in two steps.  First viewing each as a series in $L$ with coefficients which are functions of $\alpha$, these coefficients were shown to satisfy the same recurrence.
For the Dyson-Schwinger equation this $L$-recurrence is the renormalization group equation, an important equation for quantum field theories.  

The second step was 
%to consider the base case in $L$, that is 
to check that the linear coefficient in $L$ matches in the chord diagram expansion and the Dyson-Schwinger equation giving the initial conditions for the $L$-recurrence.  These coefficients are themselves series in $x$ and the proof is again done by matching recurrences.  However, this time the recurrence is more obscure, 
corresponding neither to a straightforward combinatorial decomposition nor to a standard physics identity.  Stated as an identity of weighted generating functions this equation becomes what will be numbered by \eqref{eq crazy formula} in the next subsection. In \cite{MYchord} and \cite{HYchord} we understood this formula by passing to a class of rooted trees but this class was messy and we were not able to understand the formula directly on the chord diagrams.   We will discuss this formula further, reinterpreting it in terms of rooted maps, and providing a combinatorial interpretation also at the level of rooted maps. This will show that the connection between chord diagrams and rooted maps can improve our understanding as the whole story can be formulated with one class of objects, namely rooted maps.

\subsection{Diagram parameters and binary trees}\label{subsec bin tree}

To see how the bijection $\theta$ from connected diagrams to bridgeless maps helps simplify the situation, we need to understand these additional parameters as they were originally defined.

The first thing we need is a variant of the boxed product (see Figure~\ref{fig var box} for an illustration).
\begin{definition}[Variant product for connected diagrams]\label{def var box}
  Let $C_1$ and $C_2$ be two connected diagrams and $i$ an integer between $1$ and $2|C_2|-1$.  The connected diagram $C_1 \varbox_i C_2$ is defined as
  \begin{align*}
    \RootChord i (C_2)& \text{ if $C_1$ is the one-chord diagram} \\
      \RootChord {i+\ell} \left(\DiagIns {\widehat{C_1}} i (C_2)\right) & \text{ if $C_1$ is of the form $\RootChord \ell (\widehat{C_1})$}
  \end{align*}
\end{definition}
 Decomposition according to this variant of the boxed product is known as the root-share decomposition in \cite{MYchord, HYchord}.

\begin{figure} \centering
  \includegraphics{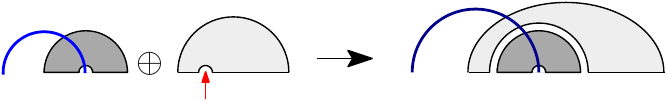}
  \caption{Schematic of the variant boxed product or root-share decomposition.}\label{fig var box}
\end{figure}

Note that this product gives the same recurrence of ordinary generating functions as the $\star$ product.  The $\star$ product is combinatorially more convenient, particularly for the asymptotic counting of \cite{CYchord}, while the $\varbox$ product is what was originally used in \cite{MYchord} and \cite{HYchord}.  The two different products clearly give a permutation $\iota$ of the set of the connected diagrams, taking the $1$-chord diagram to itself, and otherwise for a connected diagram $C=C_1\star_i C_2$, letting $\iota(C) = \iota(C_1)\varbox_i \iota(C_2)$.

The constructions below use the $\varbox$ product so as to align with the original definitions from \cite{MYchord} and \cite{HYchord}, but an analogous theory could be worked out from the $\star$ product.

The origin of the next definition is to carve out a class of rooted planar binary trees satisfying the same recurrence as comes from either connected diagram product.

\begin{definition}[Tree $\tau$(C)]
  The map $\tau$ from connected diagrams to rooted planar binary trees with labeled leaves is defined as follows.
The leaves of the tree correspond to the chords of the diagram; this correspondence is indicated by labeling the leaves by the indices of the chords in intersection order.
  \begin{itemize}
  \item The image of the one-chord diagram under $\tau$ is the rooted binary tree with one node.  This node is a leaf and is labeled $1$.
  \item Suppose $C$ is a connected chord diagram with at least 2 chords.  Write $C = C_1\varbox_k C_2$.  Let $T_1 = \tau(C_1)$ and $T_2=\tau(C_2)$.   Let $v$ be the $k$th vertex of $T_2$ in a pre-order traversal.  Let $T$ be the binary rooted tree obtained by beginning with $T_2$ and replacing $v$ with a new vertex which has the subtree rooted at $v$ as its right child and $T_1$ as its left child.  Relabel the leaves of $T$ to correspond to the same chords but as indexed in $C$, that is, the leaf $1$ from $T_1$ remains $1$, next come all the leaves of $T_2$ maintaining their relative order, and finally come all the other leaves of $T_1$ maintaining their relative order.
  \end{itemize}
\end{definition}

See Figure~\ref{fig tree} for two examples; see \cite{MYchord, HYchord} for many more examples.

It turns out that $\tau$ is one-to-one, though describing the inverse map is tricky, 
%All trees in the image of $\tau$ are binary trees with labeled leaves, with distinguished left and right children, and where all non-leaves have exactly two children.
%\nz{I replaced "binary tree" by "rooted planar binary tree" in the definition of tau, and so the previous sentence seems redundant to me -- we already stated that tau produces a "rooted planar binary trees with labeled leaves".}
%However, not every such tree is in the image of $\tau$.
and the best characterization we have for the image of $\tau$ is rather complicated (see \cite{MYchord}).
Nonetheless, $\tau$ does have some nice properties.  By construction leaves correspond to chords under $\tau$ and vertices (including leaves) correspond to intervals.  Furthermore, these trees can see the $\nu$ parameter, and the most natural decomposition of trees -- the decomposition into the root along with the left and right subtrees -- gives the formula \eqref{eq crazy formula} below.

\begin{figure} \centering
  \includegraphics{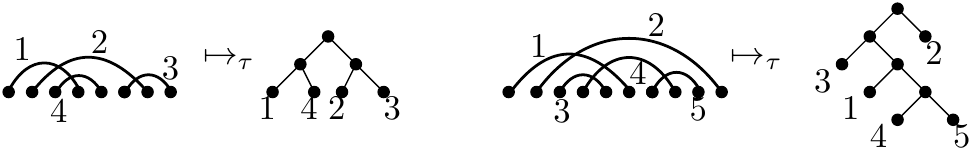}
  \caption{An example of the action of $\tau$.}\label{fig tree}
\end{figure}

Now we are ready for the original definition of $\nu$ (see \cite{HYchord} for more information on $\nu$).

\begin{definition}[Parameter $\nu(c)$]
  Let $C$ be a connected diagram and let $c$ be a chord of $C$.  Let $\nu(c)$ be the length of the path which begins at the leaf of $\tau(C)$ associated to $c$ and goes up and to the left as far as possible. If this leaf is a left child, then $\nu(c)=0$. 
\end{definition}

For the first tree in Figure~\ref{fig tree}, $\nu(1)=\nu(2)=0$, $\nu(3)=2$ and $\nu(4)=1$ agreeing with what was used in Example~\ref{eg chord params}.  For the second tree in Figure~\ref{fig tree}, $\nu(1) = 0$, $\nu(2)= 1$, $\nu(3) = 0$, $\nu(4)=0$, and $\nu(5)=3$.  At this stage it is not apparent what this parameter measures about the chord diagram.

We are finally ready for the promised mysterious formula.
To prove the main results of \cite{MYchord} and \cite{HYchord} we needed formulas which come from decomposing the binary tree associated to a diagram into its left and right subtrees.  Reversing this decomposition involves grafting the trees and shuffling \emph{some} of their labels (see \cite[Section 5]{HYchord} for this grafting, and the shuffling operation worked out in detail). We have no interpretation for the decomposition directly at the level of chord diagrams.
The formula in its more refined version is \cite[Proposition 6.10]{HYchord}:
\begin{equation}\label{eq crazy formula}
\begin{gathered}
\sum_{\substack{\|C\| = i+1\\b(C) = j+1\\\nu(b(C)) = n}} \hat w(C)A(C) = \sum_{k=1}^{i}\sum_{\ell=1}^j \binom{j}{\ell}\left(\sum_{\substack{\|D_1\|=k\\b(D_1)\geq \ell}}w(D_1)a_{d(b(D_1)), b(D_1)-\ell}A(D_1)\right)\\
\hspace{4cm} \times \ \left(\sum_{\substack{\|D_2\|=i-k+1\\b(D_2)=j-\ell+1\\\nu(b(D_2))=n-1}}\hat w(D_2)A(D_2)\right)
\end{gathered}
\end{equation}
where 
\[\hat w(C) = \prod_{m \neq b(C)}\binom{d(m)s + \nu(m) -2}{\nu(m)} = \frac{w(C)}{\binom{d(b(C))s + \nu(b(C)) -2}{\nu(b(C))}}.\]
We will give an interpretation of this equation in terms of maps in Subsection~\ref{ss:interpretation}.

Notice that the first terminal chord always has a special role to play in these quantum field theoretic chord diagram expansions.  It has its own special factor in the solutions to the Dyson-Schwinger equations \eqref{eq first dse sol} and \eqref{eq:solutionDS}.  In \eqref{eq crazy formula} on the left hand side we are ignoring the first terminal chord aside from fixing its size and index in the summation conditions.  Then in the decomposition on the right hand side the first terminal chord of the subdiagrams in the second sum remains the first terminal chord in the whole diagram and so does not contribute outside the summation conditions, but the first terminal chord of the subdiagrams in the first sum becomes a later terminal chord in the whole diagram and so it contributes a factor and more possibilities of first terminal chord must be summed over.

%which is somewhat intricate but still fits in naturally with the other map constructions we are working with. 
%
%
%Given a chord diagram $C$, the definition of the $\nu$ for each chord of $C$ divides $\tau(C)$ up into disjoint paths using all vertices.  Vertices of the tree correspond to intervals in the diagram. Thus we can interpret the definition of $\nu$ as partitioning the intervals of the diagram among the chords.  Furthermore $\nu(c)$ is one less than the number of intervals associated to $c$.  However, the results of this particular method of interval partitioning are quite odd, with no apparent combinatorial meaning, and straightforward variations on the constructions do not fix the oddness.

To step towards our new interpretation we need to associate numbers to chords in a more natural way, which the next subsection does.

\subsection{New interpretations on chord diagrams of the quantum field theoretic parameters} \label{subsec:new qft}

In this subsection, we describe an alternative notion of $\nu$-index, which we call the \textit{covering number} or \textit{$\omega$-index} and which is more meaningful at the chord diagram level, while still satisfying the above formulas. This new notion is not equivalent to the old one, so we have to establish some bijections to show that the statistics are indeed equidistributed.

%See Figure~\ref{fig weird intervals} for the interval partition worked out for the example of Figure~\ref{fig tree}.

%\begin{figure}
%  \includegraphics{Images/weird_intervals}
%  \caption{The colours and line patterns indicate the partitioning of intervals among the chords given by $\tau$ and $\nu$.}\label{fig weird intervals}
%\end{figure}

\begin{definition}[Covering number $\omega(i)$]
  Let $C$ be a connected diagram. Fix an order $c_1 < \dots < c_n$ for the chords of $C$ (for example the intersection order). Proceeding through all the chords of $C$ in that order, mark all the intervals below the current chord with the index of that chord, replacing any previous marks.  At the end of this procedure, the intervals are partitioned among the chords according to their markings.   For $i \in \{1,\dots,n\}$, let $\omega(i)$ be the number of intervals labeled by $i$ in this way, minus $1$.
\end{definition}

An example of this construction for the intersection order is given in Figure~\ref{fig good intervals}. For this diagram, we have $\omega(1)=\omega(3) = 0$, $\omega(2)=\omega(5)=1$, $\omega(4) = 2$. Note that $\nu$ and $\omega$ are not equal. 

\begin{figure} \centering
  \includegraphics{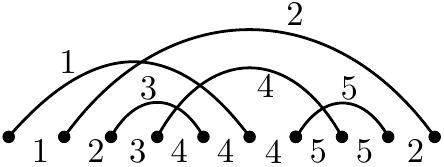}
  \caption{Covering numbers for the intersection order.}\label{fig good intervals}
\end{figure}

\begin{proposition}\label{prop nu omega equivalence}
  If we change every occurrence of $\nu$ to $\omega$, then, for the intersection order, Equation \eqref{eq crazy formula} still holds, and the function $G(x,L)$ defined by \eqref{eq:solutionDS} still solves \eqref{eq gen case}. 
\end{proposition}

The proof of the proposition directly derives from the following lemma which says that the number of diagrams with the same $\nu$ and $\omega$ vectors are equal. Moreover, this remains true if we fix the indices of the terminal chords for the intersection order.
 
\begin{lemma}\label{lem match nu}
%  Given a connected diagram $C$, let $\overline{\nu}(C) = (\nu(1), \nu(2), \ldots, \nu(|C|))$, the vector of values of $\nu$ at the chords of $C$.  Define $\overline{\nu}'(C)$ similarly.
 Let $n$ be an integer. Given an $n$-vector $\overrightarrow v = (v_1,\dots,v_n)$, % such that $v_1=0$ and $\sum_i v_i = n-1$, 
  and  a subset $S$ of $\{1, \ldots, n\}$, we denote by $A_{\overrightarrow v,S}$ (resp. $B_{\overrightarrow v,S}$) the set of connected diagrams of size $n$ such that the positions of the terminal chords for the intersection order are given by $S$, and such that $\nu(i) = v_i$ (resp. $\omega(i) = v_i$) for every $i \in \, \{1,\dots,n\}$. Then, for every vector $\overrightarrow v$ and subset $S$, the cardinal of $A_{\overrightarrow v,S}$ is the same as $B_{\overrightarrow v,S}$.
\end{lemma}

For example, there are three connected diagrams with $3$ chords and with only the last chord as a terminal chord.  These diagrams are illustrated in Figure~\ref{fig nu vs omega} with their values of $\nu$ and $\omega$ written as vectors along with the constructions to determine the vectors.  Note that for both $\nu$ and $\omega$ there is one diagram corresponding to the vector $(0, 1, 1)$ and two corresponding to $(0, 0, 2)$ but which diagrams are which is not the same.

\begin{figure}
\center
  \includegraphics[width=0.9 \textwidth]{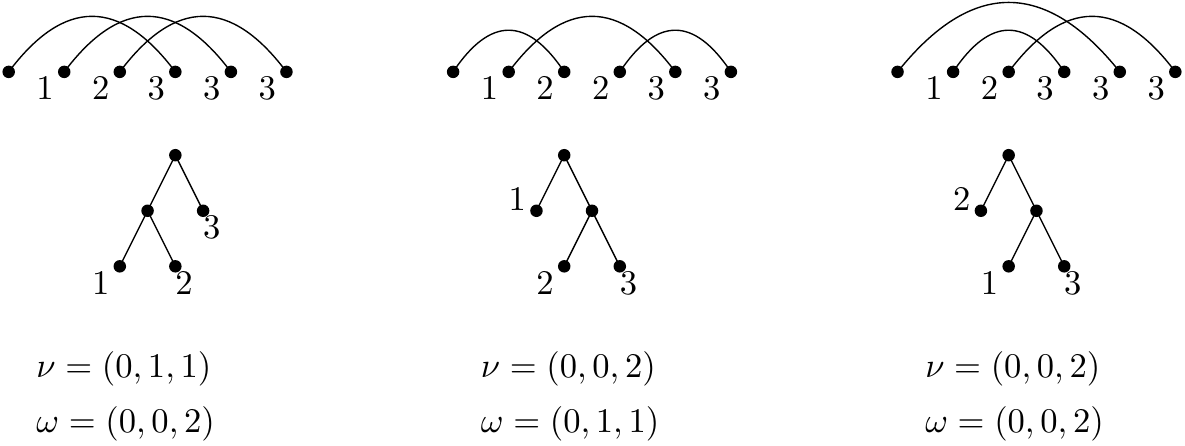}
  \caption{Connected diagrams on $3$ chords with only the last chord terminal along with some associated information.}
  \label{fig nu vs omega}
\end{figure}

\begin{proof}
The proof is by induction on the number of chords. The result clearly holds for $n=1$.

Consider two vectors $\overrightarrow u = u_1,\dots,u_{n_1}$ and $\overrightarrow v = v_1,\dots,v_{n_2}$ , and two subsets $S_1 \subseteq \{1, \ldots, n_1\}$ and $S_2 \subseteq \{1, \ldots, n_2\}$. We suppose by induction that 
$|A_{\overrightarrow u,S_1}|=|B_{\overrightarrow u,S_1}|$ and $|A_{\overrightarrow v,S_2}|=|B_{\overrightarrow v,S_2}|$.

We are going to prove that the $\nu$-indices among the diagrams of the form $C_1 \varbox_k C_2$ with $C_1 \in \, A_{\overrightarrow u,S_1}$, $C_2 \in \, A_{\overrightarrow v,S_2}$ and $k \in \, \{1,\dots,n\}$, are distributed in the same way as the $\omega$-indices among the diagrams  of the form $C'_1 \varbox_k C'_2$ with $C'_1 \in \, B_{\overrightarrow u,S_1}$, $C'_2 \in \, B_{\overrightarrow v,S_2}$ and $k \in \, \{1,\dots,n\}$.
The induction will then be shown by summing over all vectors $\overrightarrow{u}$, $\overrightarrow{v}$ and subsets $S_1,S_2$ such that $n_1 + n_2 = n$. Remark that in diagrams of the form $C = C_1 \varbox_k C_2$, the positions of the terminal chords in $C$ for the intersection order only depend on $S_1$ and $S_2$; this is why we only need to focus on the $\nu$-indices and the $\omega$-indices.

%
% .  Furthermore, if we let $C_1'$ and $C_2'$ run over all connected diagrams with the same set of terminal chords as $C_1$ and $C_2$ respectively, then every $C_1'\varbox_i C_2'$ has the same set of terminal chords as $C$.  By the inductive hypothesis for any fixed $\overline{v}_1$ and $\overline{v}_2$, $|C_1|$ and $|C_2|$-vectors respectively, the number of $C_1'$ with $\overline{\nu}(C_1')= \overline{v}_1$ is equal to the number of $C_1'$ with $\overline{\nu}'(C_1') = \overline{v}_1$ and similarly with $2$ in place of $1$.

 Fix $C_1 \in \, A_{\overrightarrow u,S_1}$ and $C_2 \in \, A_{\overrightarrow v,S_2}$. When constructing $\tau(C_1\varbox_k C_2)$ from $\tau(C_1)$ and $\tau(C_2)$, we add a new vertex along one of the leftwards paths,  so we increase exactly one $\nu$-index by $1$.  Furthermore, running over all $k$ means performing this path lengthening once at each vertex of $\tau(C_2)$. We can more precisely observe that, for every $i \in \{1,\dots,n_2\}$, there are $v_i+1$ possibilities among the choices of $k$ to increase $\nu(i)$ by $1$, since, by definition, the leftward path  starting at the leaf labeled by $i$ contains $v_i+1$ vertices. Eventually, we notice that for every vector $\overrightarrow{w}$ of the form $(u_1=0,v_1,\dots,v_{i-1},v_i+1,v_{i+1}\dots,v_{n_2},u_2,\dots,u_{n_1})$, the set $A_{\overrightarrow{w},S}$ contains exactly $v_i+1$ diagrams of the form $C_1\varbox_k C_2$, and zero such diagrams if $\overrightarrow{w}$ has a different form.
 
% that is running over $i$ means for each entry of the $\nu$ vector which corresponds to a chord of $C_2'$ we get a new diagram where the entry of the $\nu$ vector corresponding to that chord of $C_2'$ is increased by 1.

  Now consider $C = C_1'\varbox_k C_2'$ where $C'_1 \in \, B_{\overrightarrow u,S_1}$ and  $C'_2 \in \, B_{\overrightarrow v,S_2}$ are fixed. For the intersection order of $C$, every non-root chord of $C'_1$ comes after any chord of $C'_2$.  Thus, since the non-root chords of $C'_1$ are below every chord of $C'_2$, the marking of the intervals of $C'_1$ (except the first one) will overwrite the marking of the intervals delimited by the chord of $C'_2$. So, except \textit{a priori} for the root chord, the $\omega$-index associated to the chords of $C'_1$ will remain unchanged in $C$. However the $\omega$-index for the root chord is always $0$, because the label of every interval below the root chord other than the first one will be overwritten by other chords of $C$.
  
  Concerning the intervals delimited by $C_2'$, the marking will be unchanged except for the $k$th leftmost interval of $C_2'$, where the insertion of $C'_1$ occurred, splitting this interval in two. % Also, both of the pieces of the old interval $i$ will be marked by the same chord of $C_2'$ that originally marked interval $i$, because first the leftward of the new intervals will be marked by $1$, but then the marking from chords of $C_2'$ will overwrite that $1$ (along with other $1$s underneath chords of $C_2'$).  
  The marking from the  non-root chords of $C_1'$ will occur and this will overwrite all the labels inserted into interval $k$, leaving just the two ends to be marked as the $k$th interval was in $C_2'$. So if the label of the $k$th interval was $i$, then $\omega(i)$ will be increased by 1 and this is the only value of $\omega$ that changes. But, as we run over $k$, there are exactly $v_i+1$ intervals labeled by $i$ in $C'_2$.  Therefore, for every vector $\overrightarrow{w}$ of the form $(0,v_1,\dots,v_{i-1},v_i+1,v_{i+1}\dots,v_{n_2},u_2,\dots,u_{n_1})$, the set $B_{\overrightarrow{w},S}$ contains exactly $v_i+1$ diagrams of the form $C_1'\varbox_k C_2'$, and zero such diagrams if $\overrightarrow{w}$ has a different form.

Comparing the results for $C_1\varbox_k C_2$ and $C'_1\varbox_k C'_2$ over all $k$ enables us to conclude.
\end{proof}

The ideas of this last proof are closely related to some unpublished ideas of one of us along with Markus Hihn \cite{Hpersonal}.
Lemma~\ref{lem match nu} enables us to have a direct proof of Proposition~\ref{prop nu omega equivalence}.

\begin{proof}[Proof of Proposition~\ref{prop nu omega equivalence}.]
  Lemma~\ref{lem match nu} tells us that the generating functions of connected chord diagrams counted by terminal chords and $\nu$ vectors is the same with $\omega$ vectors instead.  An additional integer weight on each chord carries through the constructions with no changes.  Examples of such generating functions then, with some very particular choices of functions of these parameters, are $G(x,L)$ and the sums appearing in \eqref{eq crazy formula}, hence these formulas cannot tell the difference between $\nu$ and $\omega$.
\end{proof}

Lemma~\ref{lem match nu} also proves a conjecture of Hihn \cite[Section 3.2.1]{Hphd}, which states that the number of chord diagrams $C$ with a fixed set of terminal chords and $\nu(|C|)=m$ is equal to the number of chord diagrams with the same set of terminal chords and where the vertex in the intersection graph corresponding to the last chord has $m$ neighbors.  The last vertex having $m$ neighbours is the same as saying the last chord crosses $m$ other chords.  Furthermore in the algorithm to build $\omega$ the last chord marks all the intervals under it and the number of intervals under a chord is one more than the number of chords it crosses.  Therefore Hihn's conjecture is exactly that the number of chord diagrams $C$ with a fixed set of terminal chords and $\nu(|C|) = m$ is equal to the number of chord diagrams with the same set of terminal chords and $\omega(|C|) = m$.  This statement is a corollary of Lemma~\ref{lem match nu}.  Some of Hihn's attempts to prove the conjecture led to the arguments of \cite{Hpersonal} which were generalized into Lemma~\ref{lem match nu}.  %\ky{Check with Markus that he is happy with this credit}

Using $\omega$ in place of $\nu$ makes the parameters of \eqref{eq crazy formula} more natural, but what about the decomposition itself: what chord diagram construction builds a connected diagram out of two connected diagrams in binomially many ways.  For the $\nu$-index, the binomial coefficient counted shuffles of a subset of the labels of $\tau(C)$.  For $\omega$ the rooted maps will save the day: there we have a direct interpretation involving shuffling the edges around the root vertex, see Figure~\ref{fig:principle}.  Rooted maps are the one place where everthing becomes relatively natural.  To get there we need one last change of order on the chords.

%**********************************************************************
\subsection{Changing the ordering of the chords}\label{subsec chord orders}
%**********************************************************************

The intersection order does not induce a nice natural description when it is transposed to the set of combinatorial maps via the bijection $\theta$. In this subsection, we describe a new ordering on the chords of an indecomposable diagram for which Formulas~\eqref{eq:solutionDS} and~\eqref{eq crazy formula} still work, \textit{and} have a simple interpretation in the world of maps.

\begin{definition}[Peeling order]
The \definand{peeling order} of an indecomposable diagram $D$ is defined as follows. 
  \begin{itemize}
  \item The root chord of $D$ is the first chord in the peeling order.
  \item Remove the root chord of $D$. The result is not necessarily indecomposable. Let $D_1, D_2, \dots, D_k$ be the indecomposable diagrams we obtain from left to right.
  \item For the peeling order of $D$, after the root chord come all the chord of $D_k$ ordered recursively in the peeling order, then all the chords of $D_{k-1}$ ordered recursively, and so on.  
  \end{itemize}
\end{definition}

\begin{figure}[!ht]
\centering
\includegraphics[width = 0.8 \textwidth]{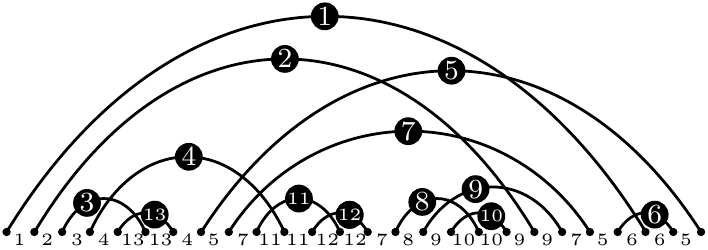}
\caption{The peeling order of a connected diagram. The covering numbers are also indicated under the intervals.}
\label{fig:peeling}
\end{figure}

An example of the peeling order is given by Figure~\ref{fig:peeling}.
Note that like the intersection order and the order by first endpoint, the peeling order extends the partial order on chords induced by the intersection graph.

Naturally, any connected diagram inherits a $\omega$-indexing from the peeling order. However, the vector distribution over all connected diagrams is not the same as for the intersection order\footnote{We have observed that a chord with a high $\omega$-index tends to be smaller in the intersection order than in the peeling order.}. Luckily, the parameters appearing in Equations~\eqref{eq:solutionDS} and~\eqref{eq crazy formula} do
 not require the exact ordering of the chords, but weaker statistics, such as the multiset of the gaps between two consecutive terminal chords. It turns out that these weaker statistics agree for the intersection and the peeling order, implying that the quantum field theory formulas still hold for the peeling order.
This also emphasizes that the gaps between terminal chords are the more natural chord diagram parameter rather than the indices of the terminal chords themselves.

\begin{proposition}\label{prop:peeling}
If we change every occurrence of $\nu$ to $\omega$, then, for the \emph{peeling order}, Equation~\eqref{eq crazy formula} still holds, and the function $G(x,L)$ defined by \eqref{eq:solutionDS} still solves \eqref{eq gen case}. 
\end{proposition}

\begin{proof}
Notably using Proposition~\ref{prop nu omega equivalence}, we saw that Formulas~\eqref{eq crazy formula} and~\eqref{eq:solutionDS} only depend on some statistics on the connected diagrams $C$ that are:
\begin{itemize}
\item[(1)] the number of chords $|C|$, the sum of the chord  weights $\|C\|$, the product \[\prod_{c\textrm{ not terminal}} 
a_{d(c),0}\] (which appears in the definition of $A(C)$ -- see Equation~\eqref{eq:defA});
\item[(2)] the position of the first terminal chord for the intersection order $b(C)$;
\item[(3)] the multiset formed by the pairs $(d(k),\omega(k))$, where $d(k)$ is the weight associated to the $k$th 
chord in the intersection order, and $\omega(k)$ its covering number for the intersection order (used to define $w(C)$ -- see Equation~\eqref{eq:defw}); 
\item[(4)] the monomial $\alpha(C) = \prod_{j=2}^\ell a_{d(t_j),t_j - t_{j-1}}$, where $t_1=b(C)<t_2<\dots<t_\ell = |C|$ lists the positions of all the terminal chords in intersection order.
\end{itemize} 
We are going to prove that these statistics are preserved diagram by diagram when we replace the intersection order by the peeling order, which is sufficient to show the proposition. 

We can first check it on an example. Let us consider the diagram of Figure~\ref{fig:intersec} where we have put a weight $2$ on chords with labels $5$, $6$, $8$ (for the intersection order) and a weight $1$ on the remaining chords. We have \begin{itemize}
\item[(1)] $|C|=13$, $\|C\|=16$,  $\prod_{c\textrm{ not terminal}} a_{d(c),0}=a_{1,0}^7$; 
\item[(2)] $b(C) = 5$; 
\item[(3)] the multiset $ \{(d(k),\omega(k))\}$ contains 4 times $(1,0)$, 5 times $(1,1)$, once $(1,2)$, once $(2,1)$, twice $(2,2)$; 
\item[(4)] $\alpha(C) = a_{1,1}^3 a_{1,2} a_{2,1} a_{2,2}$.
\end{itemize}
We can then verify that the same diagram but with the peeling order (see Figure~\ref{fig:peeling}) satisfies the same equalities. However remark that the positions of the terminal chords differ between the peeling order and the intersection order (these positions are given by $5,6,7,9,10,12,13$ for the peeling order, and by $5,6,8,9,11,12,13$ for the intersection order).

\begin{figure}[!ht]
\centering
\includegraphics[width = 0.8 \textwidth]{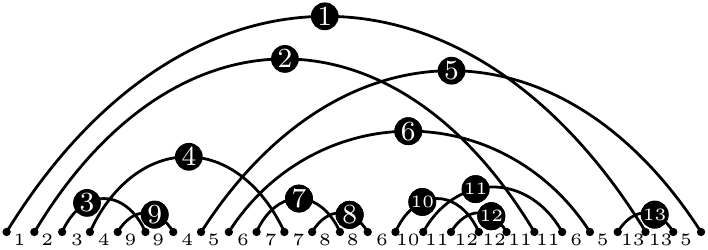}
\caption{The intersection order version of the diagram of  Figure~\ref{fig:peeling}.}
\label{fig:intersec}
\end{figure}

Return now to the proof.
Obviously, the statistics listed in (1) do not depend on the order.

As for the position of the first terminal chord  given by (2), we can observe that the intersection order and the peeling order coincide for the first chords until the first terminal chord. Indeed, in both cases, after putting in first position the root chord and removing it, the first diagram we recursively sort is either the topmost connected component $C_\uparrow$ (for the intersection order), or the rightmost indecomposable diagram $D_\rightarrow$ (for the peeling order). The diagram $C_\uparrow$  is included in $D_\rightarrow$ and will be peeled first in $D_\rightarrow$ because the connected components below $C_\uparrow$  are to the left of the rightmost endpoint of $C_\uparrow$ (so they will appear at some point of the peeling of $D_\rightarrow$ to the left of what remains of $C_\uparrow$). Thus, the position of the first terminal chord remains the same for the intersection and peeling order.

Now let us consider the multiset $\{(d(k),\omega(k))\}_k$ described by (3). 
%The labels of the chords are here irrelevant, since we consider them as a whole in a multiset.\ky{I think this is a bit confusing} Moreover, 
Remark that the covering number of a chord $c$ will only depend on the chords above/below
%\ky{the above/below language is not so clear since it could mean greater than or less than in the order in question or above or below in the diagram itself} 
$c$ in the diagram, and the chords intersecting $c$. But both for  intersection and peeling order, a chord $c_\downarrow$ which is below a chord $c_\uparrow$ will satisfy $c_\uparrow < c_\downarrow$, while a chord $c_\leftarrow$ intersecting from the left a chord $c_\rightarrow$ will satisfy $c_\leftarrow < c_\rightarrow$.
Therefore, the covering number associated to any chord will remain the same for the intersection and the peeling order, hence the equality of the multisets.

The point (4) is the most delicate equality to establish. To remove the ambiguity, let $\alpha_{inter}(C)$ be the version of $\alpha(C)$ for the intersection order, and $\alpha_{peel}(C)$ be the one for the peeling order. We are going to prove by induction that $\alpha_{inter}(C)=\alpha_{peel}(C)$ for any connected diagram $C$. Since the base case is clear, we assume that $C$ has at least $2$ chords. Let $C_1$, $C_2$, $i$ be such that $C = C_1 \varbox_i C_2$. We assume that $C_1$ is not reduced to one chord, since it is easy to conclude by induction in that case.

First we observe that, in the intersection order, each non-root chord of $C_1$ is after any chord of $C_2$ (by definition). So if $C_2$ exactly contains $j$ terminal chords, then the terminal chords with positions $t_1,t_2,\dots,t_j$ in $C$ are in $C_2$ (diagram in which the terminal chords have positions $t_1-1,t_2-1,\dots,t_j-1$), and the other ones are in $C_1$. Moreover, the last chord of a connected diagram for the intersection order is terminal, hence $t_j = |C_2|+1$. Additionally, if $t'_1=b(C_1),\dots,t'_k$ denote the positions of the terminal chords in $C_1$, we can check that $t_{j+p} = t'_p + |C_2| $ for $p \in \, \{1,\dots,k\}$. Taking all this into account, we obtain
\begin{align*}
\alpha_{inter}(C) & = \alpha_{inter}(C_2) \times a_{d(t_{j+1}),t_{j+1}-t_j} \times \alpha_{inter}(C_1) \\ & = a_{d(b(C_1)),b(C_1)-1} \  \alpha_{inter}(C_1) \  \alpha_{inter}(C_2).
\end{align*}

Now let us consider the peeling order. Let $D$ be the diagram $C_1$ with its root chord removed. When we remove the root chord of $C$, the diagram $D$ is left somewhere in the diagram $C_2$. When we continue to peel $C$, the chords of $D$ will remain unconsidered until the point where $D$ appears as one of the indecomposable diagrams $D_1, D_2, \dots, D_k$. There are then two possibilities: either $D=D_k$ and then the chord preceding the first chord of $D$ for the peeling order is a chord going over $D$ and ending at the rightmost point of the diagram; or $D = D_j$ with $j < k$ and then the chord preceding the first chord of $D$ is the last chord of $D_{j+1}$. In any case, the chord preceding the first chord of $D$ is terminal, so its position should be of the form $t_q$. Thus, if $t'_1,\dots,t'_k$ denote the positions of the terminal chords of $D$, then  $t_{q+r} = t'_r + t_q$, for $r \in \{1,\dots,k\}$. We have then
\[\prod_{j=q+1}^{q+k} a_{d(t_{j+1}),t_{j+1}-t_j}  = a_{d(t_{q+1}),t_{q+1}-t_q} \  \alpha_{peel}(D) = a_{d(b(D)),b(D)} \  \alpha_{peel}(D).\]
Furthermore, $C_1$ differs from $D$ just by a root chord insertion, hence we have $\alpha_{peel}(D)=\alpha_{peel}(C_1)$ so that 
\[\prod_{j=q+1}^{q+k} a_{d(t_{j+1}),t_{j+1}-t_j}  =  a_{d(b(C_1)),b(C_1)-1} \  \alpha_{peel}(C_1).\]

Compare now the peelings of $C$ and $C_2$. We can process them in parallel, except that at some point in the peeling of $C$, we have to treat the subdiagram $D$. After finishing the peeling of $D$, we can resume the peeling of $C$ and $C_2$ in parallel. Thus, since the chord visited just before $D$ has label $t_q$, and $D$ has $k$ terminal chords, the set of gaps between two terminal chords of $C_2$ is constituted by $t_2 - t_{1}, t_3 - t_2, \dots, t_q - t_{q-1}$ (occurring in $C$ before visiting $D$), then $t_{q+k+1}-|D|-t_q$ (in $C_2$ we do not visit $D$, so we have to subtract $|D|$ from the labels $\geq t_q + |C_2|$ of $C$  to recover the labels of $C_2$), and finally $t_{q+k+2} - t_{q+k+1},\dots,t_\ell-t_{\ell-1}$ (occurring in $C$ after $D$). Note that $t_{q+k}=t'_k + t_q$, which is also equal to $|D|+t_q$ since the last chord is always terminal. Therefore we have
\[\alpha_{peel}(C_2) = \prod_{j=1}^q a_{d(t_{j+1}),t_{j+1}-t_j} \times \prod_{j=q+k+1}^\ell a_{d(t_{j+1}),t_{j+1}-t_j}\]
so that
\[\alpha_{peel}(C) = a_{d(b(C_1)),b(C_1) - 1} \  \alpha_{peel}(C_1) \  \alpha_{peel}(C_2).\]
We then conclude that $\alpha_{inter}(C) = \alpha_{peel}(C)$ by the induction hypothesis.
\end{proof}

%**********************************************************************
\subsection{Restating the quantum field theory formulas in terms of maps}\label{subsec new dfs}
%**********************************************************************

Now let us think about how all the previous work clarifies the situation when the diagrams are transformed into combinatorial maps under $\theta$. 

%The key to translating over these parameters is the spanning tree coming from the rightmost DFS -- the same tree used in the Bridge First Labeling.

The key is that here the orientation of the map given by the \textit{rightmost DFS} (Depth First Search) of the map.  The spanning tree is the same as in the Bridge First labeling but the labeling is quite different.

\begin{definition}[Rightmost DFS]
The principle of the rightmost DFS is the following. Starting from the root, we explore the map as far as possible by choosing at each newly visited vertex the nearest half-edge in clockwise order. If the other associated half-edge belongs to an already visited vertex, we backtrack. We stop once every edge has been visited. 
\label{def:rdfs}
\end{definition}

This map traversal naturally induces an orientation of the edges of the map, as illustrated by Figure~\ref{fig:rdfs}. 

\begin{figure}[!ht]
\centering
\includegraphics[scale=1.9]{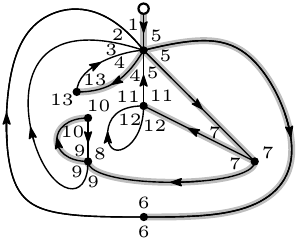}
\caption{The rightmost DFS of a map and its associated statistics.}
\label{fig:rdfs}
\end{figure}

We now give an equivalent of the $\omega$-index for maps $M$.  The principle is illustrated by Figure~\ref{fig:dfslabel}.

\begin{definition}[DFS-labeling of a map] We are going to label the corners of a map $M$ with integers $1,\dots,|M|$, using the orientation induced by the rightmost DFS.
 We start with the corner following the root, whose label is $1$. Suppose that the current corner is labeled by $i$, and the next corner around the vertex in the counterclockwise order is not labeled. If the edge separating these two corners is ingoing, then we label the second corner by $i+1$; otherwise, the edge is outgoing, and we label the corner by $i$. Once all corners around the current vertex have been labeled, we go to the vertex which has been visited next during the rightmost DFS. Around this vertex, there is only one ingoing edge coming from the spanning tree induced by the rightmost DFS --- it is the first edge that enabled the visit of this vertex. We then label the corner following this edge by the next available label, and continue the procedure. We stop when every corner is labeled. 
 \end{definition}
 
 The reader can refer again to Figure~\ref{fig:rdfs} for an example. 
 
Similarly to diagrams, we can define $\omega(k)$ for maps as the number of corners carrying the label $k$ (minus $1$). However it will be more convenient to define $\omega$ for edges. Thus, to each edge $e$, the integer $\omega(e)+1$ is the number of corners carrying the same label as the corner that is clockwisely adjacent to the ingoing part of $e$. Equivalently, $\omega(e)$ is the number of outgoing edges between the ingoing part of $e$ and the next ingoing half-edge after $e$ in the clockwise order. For example, the value of $\omega$ applied to the root of the map in Figure~\ref{fig:rdfs} is $2$, since there are three labels $5$.

 \begin{figure}[!ht]
\centering
\includegraphics[width=\textwidth]{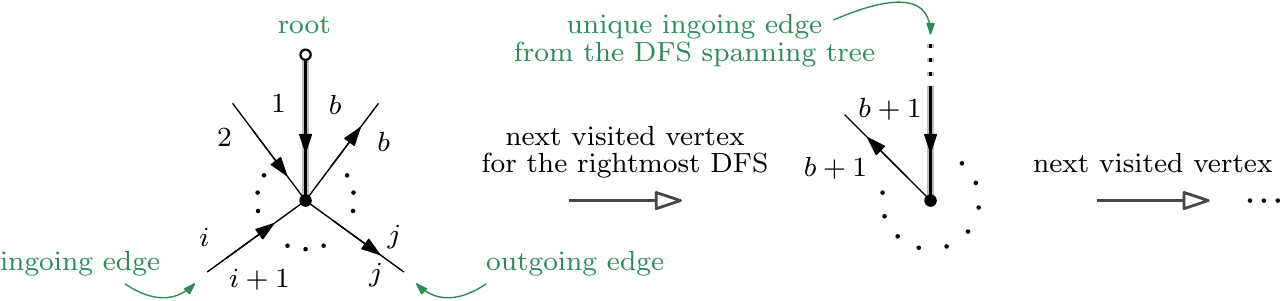}
\caption{DFS-labeling procedure}
\label{fig:dfslabel}
\end{figure}

We can now describe how the statistics from the QFT formulas translate to maps.

\begin{proposition} Under the bijection of Section~\ref{sec:bijection}, the parameters of \eqref{eq:solutionDS} are transferred as indicated by Table~\ref{tab:transfer}.
\label{prop:translation}
\end{proposition}
\begin{table}
\centering
\begin{tabular}{c|c} \hline
 \begin{minipage}{0.45\textwidth} \centering \vspace*{2pt} Parameters in connected chord diagrams
 \vspace*{2pt} \end{minipage} & Parameters in bridgeless maps \\ \hline\hline
chords & edges \\ \hline
terminal chords & \begin{minipage}{0.45\textwidth} \centering \vspace*{2pt}\vspace*{2pt} vertices; edges in the spanning tree induced by the rightmost DFS
 \vspace*{2pt}
 \end{minipage}  \\ \hline
 \begin{minipage}{0.45\textwidth} \centering \vspace*{2pt}  position $b(C)$  of the first terminal chord \vspace*{2pt}
  \end{minipage}
   & \begin{minipage}{0.45\textwidth} \centering \vspace*{2pt} number of ingoing edges (for the rightmost DFS)  incident to the root vertex 
 \vspace*{2pt}
 \end{minipage}
 \\ \hline 
 \begin{minipage}{0.45\textwidth} \centering \vspace*{2pt} gap $t_j-t_{j-1}$ between the $(j-1)$th and the $j$th terminal chords \vspace*{2pt}
  \end{minipage}
  &
 \begin{minipage}{0.45\textwidth} \centering \vspace*{2pt}
 number of ingoing edges (for the rightmost DFS) incident to the vertex which has been visited at position $j$ in the rightmost DFS  \vspace*{2pt}
 \end{minipage} \\ \hline
 \begin{minipage}{0.45\textwidth} \centering \vspace*{2pt}  $\omega$-index of the $k$th chord  \\  for the peeling order
 \vspace*{2pt}
  \end{minipage}
& \begin{minipage}{0.45\textwidth} \centering \vspace*{2pt} number of corners labeled by $k$ \\  for the DFS-labeling procedure \\ minus 1
 \vspace*{2pt}
 \end{minipage}
 \\ \hline 

%$\nu$-vector &  \begin{minipage}{0.5\textwidth} \centering  \vspace*{2pt} something which can simply described in terms of the ingoing/outgoing edges for the rightmost DFS 
%\end{minipage}
\end{tabular}
\begin{caption}{
How parameters intervening in the QFT formulas transfer from diagrams to maps.}
\label{tab:transfer}
\end{caption}
\end{table}
This proposition can be in particular verified by comparing Figures~\ref{fig:peeling} and~\ref{fig:rdfs}, whose map and diagram are in bijection through $\phi$.

The most striking correspondence is the one between the terminal chords and the vertices of a map. First of all, it implies that the original QFT formulas can be expressed in terms of bridgeless maps counted with respect to edges and vertices, which are admittedly more natural than connected diagrams and terminal chords.
It also again emphasizes that the gaps not the terminal chords themselves are the right parameter.  Moreover, all the asymptotic results of \cite{CYchord} translate over to asymptotics about vertices of bridgeless maps.  For example, it proves that the number of vertices in a random bridgeless map asymptotically obeys a Gaussian law of mean $\sim \ln n$. 

\begin{proof}[Proof of Proposition~\ref{prop:translation} (sketch)]
The proof is a simple induction on (not necessarily bridgeless) maps $M$. It uses the fact that $\theta$ can be extended to $\phi$ (see Theorem \ref{theo:restriction}). Indeed, it is sufficient to consider $M$ under all possible forms (map reduced to one edge; $M=\MapIns {M_2} 1 (M_1)$; $M = \RootEdge i (M')$) and confront it to its image under $\phi$ (respectively the diagram reduced to one chord; $\phi(M)=\MapIns {\phi(M_2)} 1 (\phi(M_1))$; $\phi(M) = \RootEdge i (\phi(M'))$).

The proof is not difficult, but it requires a tedious checking through all para\-meters. All the necessary ideas are depicted in Figure~\ref{fig:translation}. \end{proof}

 \begin{figure}[!ht]
\centering
\includegraphics[width=\textwidth]{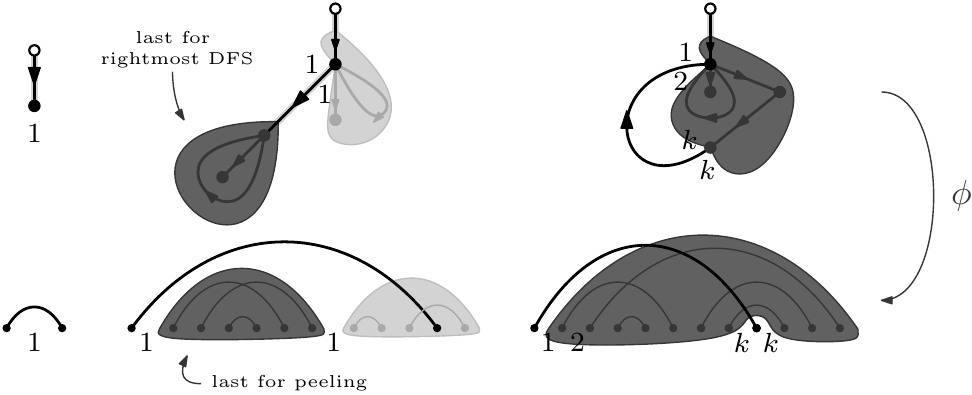}
\caption{How the statistics evolve from maps to indecomposable diagrams}
\label{fig:translation}
\end{figure}

Thanks to Propositions~\ref{prop:peeling} and~\ref{prop:translation}, we can rewrite the formulas we described in Subsection~\ref{ss:qftcontext} in terms of maps, offering a new viewpoint on these equations. In particular, Equation~\eqref{eq:solutionDS} can be written under the following form.

\begin{corollary}
Let $F_k(\rho)$ be of the form $\sum_{i\geq 0} a_{k, i} \rho^{i-1}$, and $s$ be a positive integer parameter. The Dyson-Schwinger equation
%The main result of \cite{HYchord}, generalizing \cite{MYchord}, consists of solving a family of Dyson-Schwin\-ger equations in terms of a sum over weighted connected chord diagrams:
%
%\begin{theorem}[Hihn, Yeats~\cite{HYchord}] For a positive integer $s$, we define the Dyson-Schwinger equation
\begin{equation*}
G(x, L) = 1 - \sum_{k \geq 1}x^kG(x, \partial_{-\rho})^{1-sk}(e^{-L\rho}-1)F_k(\rho)_{\rho=0}
\end{equation*}
has for solution 
\[G(x,L) = 1 - \sum_M \left( \sum_{i =  1}^{\rid(M)}  a_{d(root(M)), \rid(M)-i} \frac{(-L)^i}{i!} \right) w(M)
A(M) x^{\|M\|},\]
where the sum runs over all bridgeless maps $M$, carrying a positive integer weight $d(e)$ on every edge $e$.  As for the other parameters, $\rid(M)$ is the number of ingoing edges induced by the rightmost DFS (see Definition~\ref{def:rdfs}); $\|M\|$ is the sum of the edge weights;
\begin{equation}
w(M) = \prod_{e\textrm{ edge }\in M}\binom{d(e)s + \omega(e) - 2}{\omega(e)};
\label{eq:wmap}
\end{equation}
$\omega(e)$ is the number of outgoing edges between the ingoing part of $e$ and the next ingoing half-edge after $e$ in the clockwise order;
\begin{equation} A(M) = \prod_{\substack{e \text{ not in the}\\ \text{DFS spanning tree}}}a_{d(e), 0}\prod_{\substack{e \neq \text{root and in the}\\ \text{DFS spanning tree}}} a_{d(e), \id(\ve(e))};
\label{eq:amap}
\end{equation}
and $\id(\ve(e))$ is the number of ingoing edges around the vertex pointed by the edge $e$.
\label{cor:allinmaps}
\end{corollary}

%*******************************************
\subsection{A new combinatorial interpretation of a quantum field theoretic formula}
\label{ss:interpretation}
%*******************************************

As an application of the map interpretation of the solution of the previous Dyson-Schwinger equations, we are going to describe an interpretation of Equation~\eqref{eq crazy formula} at the map level.  Then with Corollary~\ref{cor:allinmaps} all steps and tools can be understood on the same objects namely combinatorial maps.  Recall that this equation was in the core of the proof of the papers \cite{MYchord,HYchord} but the proof passed to rooted trees in an obscure way and was never understood at the level of chord diagrams.
It can be reformulated in terms of maps as follows.

\begin{proposition} Let $G_d(x,c)$ and $\widehat G_d(x,c)$ be the weighted generating functions 
\[G_d(x,c) = \sum_{\substack{M\text{ bridgeless map}\\\text{with a weight }>0\\\text{on each edge}\\ \text{with }\rid(M)=d}} w(M)\,A(M)\,x^{\|M\|}\,c^{\omega(root(M))} ,\]
\[\widehat G_d(x,c) = \sum_{\substack{M\text{ bridgeless map}\\\text{with a weight }>0\\\text{on each edge}\\ \text{with }\rid(M)=d}} \hat w(M)\,A(M)\,x^{\|M\|}\,c^{\omega(root(M))}, \]
where $\rid(M)$ is the number of ingoing edges induced by the rightmost DFS incident to the root vertex, $\|M\|$ is the sum of the edge weights,  $\omega(root(M))$ is the number of outgoing edges between the root and the next ingoing edge for the clockwise order, $w(M)$ and $A(M)$ are respectively defined by~\eqref{eq:wmap} and~\eqref{eq:amap}, and
\begin{equation*}
\hat w(M) = \prod_{\substack{e\textrm{ edge }\in M\\\text{\textbf{different from the root}}}}\binom{d(e)s + \omega(e) - 2}{\omega(e)}.
\end{equation*}
Then for $d \geq 2$,
\begin{equation}
\widehat G_{d}(x,c)= c \, \sum_{\substack{d_1\geq 1, i \geq 1\\d_1 + i = d}}\sum_{d_2 \geq i} \binom{d_1 + i - 1} i \widehat  G_{d_1}(x,c) \,
dec_{d_2,i}(x) \, G_{d_2}(x,1), 
\label{eq:newcrazyformula}
\end{equation}
where $dec_{d_2,i}(x) = \sum_{k \geq 1} a_{k,d_2-i} \,(1-x)\, x^{k-1}$.
\label{prop:translatedcrazyformula}
\end{proposition}

\begin{proof}\textbf{1. Principle.}
This proof is rather complex and will be divided in several parts. The idea is to interpret the right side of Equation~\eqref{eq:newcrazyformula} as the combination of two bridgeless maps that we shuffle at the level of their root vertices.

More precisely, we are going to consider two bridgeless maps $M_1$ and $M_2$, where the numbers of ingoing edges (for the rightmost DFS) incident to the root vertex are respectively $d_1$ and $d_2$ and we fix any $i \in \{1,\dots,d_2\}$. Roughly speaking, we are going to split the root vertex of $M_2$ in $d_2$ pieces containing each one of them an ingoing edge, then select the first $i$ such pieces and glue them on the root vertex on $M_1$. Meanwhile, the root of $M_2$ will be inserted at the corner just to the right of the root. This principle is illustrated by Figure~\ref{fig:principle}.

 \begin{figure}[!ht]
\centering
\includegraphics[width=\textwidth]{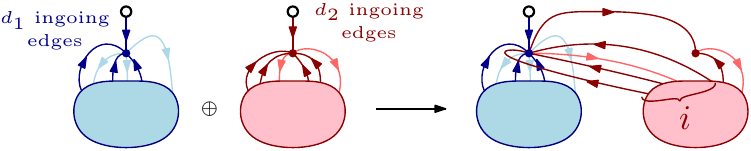}
\caption{Interpretation of Equation~\eqref{eq:newcrazyformula} as the combination of two bridgeless maps.}
\label{fig:principle}
\end{figure}

The series $dec_{d_2,i}(x) = \sum_{k \geq 1} a_{k,d_2-i} \,(1-x)\, x^{k-1}$ is introduced to deal with the fact that the root of $M_2$ is no longer the root after the operation, and so $A(M_2)$ has been modified.

\textbf{2. Splitting the root of $\boldsymbol{M_2}$.} 
The half-edges incident to the root of $M_2$ can be listed in the counterclockwise order as
\[i_1,(o_{2,1},\dots,o_{2,j_2}),i_2,(o_{3,1},\dots,o_{3,j_3}),i_3,\dots,(o_{d_2,1},\dots, o_{d_2,j_{d_2}}),i_{d_2}=root(M_2),\]
where $i_1,\dots,i_{d_2}$ are the $d_2$ ingoing edges incident to the root vertex, $i_{d_2}$ is the root of $M_2$, and $(o_{k,1},\dots,o_{k,j_k})$ is the sequence (potentially empty) of outgoing edges preceding $i_k$. Note that $j_k=\omega(i_k)$ for every $k \in \{1,\dots,d_2\}$. 

We split the root vertex of $M_2$ into $d_2$ smaller vertices $v_1,\dots,v_k$ such that the incident half-edges of $v_k$ are $o_{k,1},\dots,o_{k,j_k},i_k$. Let us denote the resulting map $\widehat M_2$. Remark that $\widehat M_2$ is still connected since we can still carry out a DFS with the same orientation (maybe not in the same order, but if we need to backtrack to the root vertex to follow an outgoing edge, this edge is necessarily attached to an ingoing edge which has been previously visited). The process is shown in Figure~\ref{fig:splitM2}.

 \begin{figure}[!ht]
\centering
\includegraphics[scale=2]{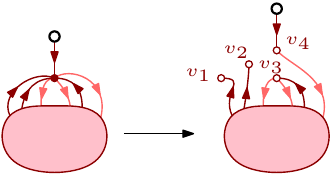} \hfill\includegraphics[scale=2]{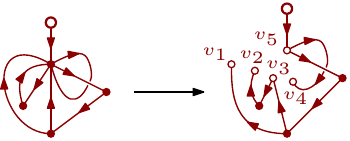}
\caption{Typical splitting of $M_2$, along with an example.}
\label{fig:splitM2}
\end{figure}

\textbf{3. Defining the map $\boldsymbol M$.}
We are going to merge the vertices $v_1,\dots,v_i$ of $\widehat M_2$ with the root vertex of $M_1$ at some particular locations. These locations are just inside the corners that counterclockwisely follow an ingoing edge. (Thus there are $d_1$ such corners.) Figure~\ref{fig:locations} illustrates that.

 \begin{figure}[!ht]
\centering
\includegraphics[scale=2]{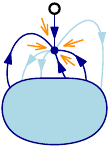} \hspace{0.25\textwidth} \includegraphics[scale=2]{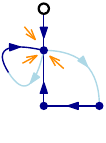}
\caption{Corners after the ingoing edges, along with an example.}
\label{fig:locations}
\end{figure}

We fix now a subset $S$ of these locations, multiplicity allowed, of size $i$. (Since we authorize multiple occurrences of the same location, there are $\multiset {d_1} {i} = \binom{d_1 + i - 1} i$ such subsets.) Then we glue $v_1$ at the first\footnote{\textit{First} means here \textit{first in the counterclockwise order, if we start from the root}.} corner given by $S$, putting $i_1$ in last. We similarly glue $v_2$ in the second position, then $v_3$, and so on and so forth, finishing by $v_i$. We glue back $v_{i+1},\dots,v_{d_2}$ as they were before in $M_2$.

Moreover, we attach the root of $\widehat{M_2}$ as a non-root edge just in the corner following the root of $M_1$ in the clockwise order.

The resulting map is denoted $M$. Complete examples are given by Table~\ref{tab:combin}. Note that when $i = d_2$, the root of $M_2$ becomes a loop.

\begin{table}[!ht]
\begin{center}
\begin{tabular}{|c|c|c|c|} \hline
i &$M_1$ with $S$ & $M_2$ & Resulting $M$
\\
\hline   \vspace*{-.37cm} & & &\\
\hline
3 &
\begin{minipage}{0.15\textwidth}
\vspace*{.2cm}
\begin{center}
\includegraphics[scale=2]{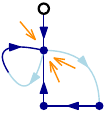} 
\end{center}
\vspace*{.01cm}
\end{minipage}
&
\begin{minipage}{0.15\textwidth}
\vspace*{.2cm}
\begin{center}
\includegraphics[scale=2]{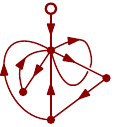} 
\end{center}
\vspace*{.01cm}
\end{minipage}
&
\begin{minipage}{0.45\textwidth}
\vspace*{.2cm}
\begin{center}
\includegraphics[scale=2]{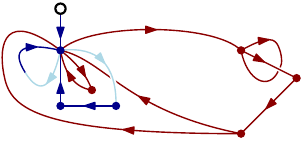} 
\end{center}
\vspace*{.01cm}
\end{minipage}
\\
\hline
5 &
\begin{minipage}{0.15\textwidth}
\vspace*{.2cm}
\begin{center}
\includegraphics[scale=2]{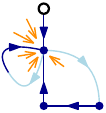} 
\end{center}
\vspace*{.01cm}
\end{minipage}
&
\begin{minipage}{0.15\textwidth}
\vspace*{.2cm}
\begin{center}
\includegraphics[scale=2]{exM2} 
\end{center}
\vspace*{.01cm}
\end{minipage}
&
\begin{minipage}{0.45\textwidth}
\vspace*{.2cm}
\begin{center}
\includegraphics[scale=2]{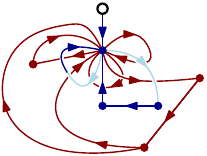} 
\end{center}
\vspace*{.01cm}
\end{minipage}
 \\  \hline
\end{tabular}
\end{center}
\caption{Examples of combinations between two bridgeless maps $M_1$ and $M_2$.}
\label{tab:combin}
\end{table}

\textbf{4. How the parameters evolve.} First of all, the weights on the edges do not change during the operation, so $\|M\| =\|M_1\| + \|M_2\|$. 

One outgoing edge was added to the right of the root of $M$ (which was the root of $M_2$), so the number of outgoing edges of $M$ between the root and the next ingoing edge in the clockwise order has been increased by $1$ compared to $M_1$. In other words, $\omega(root(M))= \omega(root(M_1)) + 1$. Additionally, since each vertex $v_1,\dots,v_i$ has one ingoing edge, we have $\rid(M)=\rid(M_1)+i=d_1 + i.$

Concerning $A(M)$, we remark that it compiles every factor of $A(M_1)$ and $A(M_2)$, and the factor associated to the root of $M_2$ (which is no longer a root in $M$). There are two possibilities here: either $i < d_2$, and in that case, the root of $M_2$ belongs to the DFS spanning tree of $M$, and because we have removed $i$ ingoing edges to the root vertex of $M_2$, this factor is $a_{d(root(M_2)),d_2 - i}$; or $i=d_2$, and the root vertex of $M_2$ is merged with the root vertex of $M_1$, implying that $root(M_2)$ is not in the spanning tree of $M$, hence the factor is $a_{d(root(M_2)),0}$. In every case, we have $A(M) = a_{d(root(M_2)),d_2 - i}A(M_1)A(M_2)$.

As for $\hat w(M)$, observe that $\omega(e)$ is invariant for every edge $e$ different from the root of $M$. We have for that purpose split the root of $M_2$ in pieces which preserve the number of outgoing edges before an ingoing edge. Consequently, $\hat w(M)= w(M_1) \hat w(M_2)$.

It is then relatively easy to see that the weighted generating function of maps $M$ (potentially with multiplicity) produced by the combinations of every pair of maps $M_1$ and $M_2$, with respectively $d_1$ and $d_2$ ingoing edges incident to the root vertex, is given by the right side of \eqref{eq:newcrazyformula}. The only subtlety here is the incorporation of $a_{d(root(M_2)),d_2 - i}$ which depends on the decoration of the root of $M_2$. To deal with this, we remark that the weighted generating function of maps $M_2$ where we have removed the weight of the root is given by $\dfrac{G_{d_2}}{\sum_{k \geq 1} x^k}= \dfrac{G_{d_2}}{\frac x {1-x}}$.
 Then, to recover the weight of the root of $M_2$ along with $a_{d(root(M_2)),d_2 - i}$, we have to multiply the previous series by $\sum_{k \geq 1} a_{k,d_2 - i}x^k$, which gives $dec_{d_2,i}(x)G_{d_2}$.
 
 Thus, to prove Equation~\eqref{eq:newcrazyformula}, it just remains to show that the construction is bijective, which is the purpose of the last point.
 
\textbf{5. Recovering $\boldsymbol {M_1}$, $\boldsymbol {M_2}$ and $\boldsymbol i$.} 
Given a map $M$, we are going to construct two maps $M_1$ and $M_2$ whose combination gives $M$. The process is illustrated by Figure~\ref{fig:unmerging}.

 \begin{figure}[!ht]
\centering
\includegraphics[width=\textwidth]{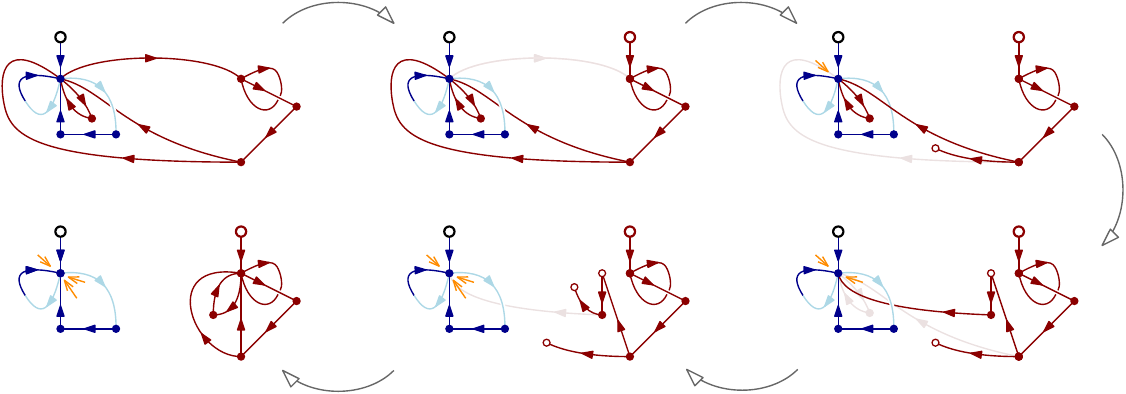}
\caption{How to recover $M_1$ and $M_2$.}
\label{fig:unmerging}
\end{figure}

We start by detaching the edge clockwisely following the root edge and making it a root. This will be the root of the map $M_2$. We are going now to successively detach edges which are incident to the root vertex of $M$ until we obtain two separate maps.

To do so, we run a rightmost DFS of the map that starts from the root of $M_2$. Whenever we return to the root vertex of $M$, we detach the corresponding ingoing edge along with the whole sequence of outgoing edges that clockwisely follow it. We repeat this until $M_2$ forms a new connected component. At this point, we glue every detached vertex to the root vertex of $M_2$, in the same order that these vertices were attached to the root vertex of $M$.

We thus prove that the combination procedure is bijective.
\end{proof}

Over all, the message here is that the map interpretation is helpful and more natural for the chord diagram expansions in quantum field theory of \cite{MYchord, HYchord, CYchord}.  Some of these improvements are manifestly simple such as the reinterpretation of terminal chords as vertices.  Others, such as the formula of this section, are considerably more intricate.  Keep in mind, however, that the original proof of this formula was also very intricate and went though subtle auxiliary objects, and part of the complexity exhibited here is in proving the connection between the two approaches as in Subsection~\ref{subsec chord orders}, rather than due to the new approach itself.

%%%%%%%%%%%%%%%%%%%%%%%%%%%%%%%%%%%%%%%%%%%%%%%%%
%%%%%%%%%%%%%%%%%%%%%%%%%%%%%%%%%%%%%%%%%%%%%%%%%
\section{New interpretation of Arquès and Béraud's functional equation}
\label{sec:arques-beraud}
%%%%%%%%%%%%%%%%%%%%%%%%%%%%%%%%%%%%%%%%%%%%%%%%%
%%%%%%%%%%%%%%%%%%%%%%%%%%%%%%%%%%%%%%%%%%%%%%%%%

\subsection{Statement of the equation and implications}
\label{subsec:arques-beraud-statement}

In \cite{ABmaps}, Arquès and Béraud studied the two-variable generating function
\begin{multline*}
B(z,u) = u + z(u + u^2) + z^2(3u + 5u^2 + 2u^3) + %\\
z^3(15u + 32u^2 + 22u^3 + 5u^4)+\dots
\end{multline*}
counting rooted maps with respect to edges\footnote{Note that our rooting convention for maps allocates one additional (dangling) edge relative to Arquès and Béraud's convention, explaining the seeming shift by a factor of $z$.} ($z$) and vertices ($u$), and proved that it satisfies the following simple functional equation:
%(sometimes referred to as a \textit{generalized Dyck equation})
\begin{align}
\label{eq:ABeq1}
B(z,u) &= u + zB(z,u)B(z,u+1)
\end{align}
Arquès and Béraud showed how to derive \eqref{eq:ABeq1} algebraically starting from another functional differential equation 
%%\begin{equation}
%% \label{eq:ABeq2}
%% B(z,u) = u + zB(z,u)^2 + zB(z,u) + 2z^2 \frac{\partial}{\partial z} B(z,u)
%% \end{equation}
which they established through a root edge decomposition of maps on oriented surfaces of arbitrary genus (a refinement of the basic analysis we described in Section~\ref{ss:btwn-id-m}).
%% \begin{equation}
%% \label{eq:ABeq2}
%% B(z,u) = u + zB(z,u)^2 + zB(z,u) + 2z^2 \frac{\partial}{\partial z} B(z,u)
%% \end{equation}
Later, Cori~\cite{Cori2009} gave an alternative proof of \eqref{eq:ABeq1} that made use of Ossona de Mendez and Rosenstiehl's bijection (henceforth, the ``OMR bijection'') between combinatorial maps and indecomposable involutions \cite{OdMRencoding},
which sends vertices of a map to \textit{left-to-right maxima} of the corresponding indecomposable involution.
Speaking in terms of chord diagrams, 
left-to-right maxima correspond to \textit{top chords}: that is, chords which are not below another chord.
For example, the number of top chords in the diagrams of size $3$ of Tables~\ref{tab:smallex} and \ref{tab:smallex2} are respectively % $1$, $2$,
3, 3, 2, and 2 for the connected diagrams, and 1, 1, 1, 2, 2, and 2 for the disconnected diagrams.

In the following section, we give a direct bijective interpretation of \eqref{eq:ABeq1} on indecomposable chord diagrams.
Besides its intrinsic interest, this bijection has the useful property that it restricts to connected diagrams to verify a modified functional equation:
\begin{equation}
\label{eq:connected-simple}
C(z,u) = u + zC(z,u)(C(z,u+1) - C(z,1))
\end{equation}
By Theorem~\ref{theo:simplebijection}, we know that $C(z,1)$ is also the generating function for bridgeless maps counted by number of edges, and we will use this fact later to derive an interesting application to the combinatorics of lambda calculus (Section~\ref{subsec:lambda}).
On the other hand, we do not see an obvious interpretation of the $u$ parameter of \eqref{eq:connected-simple} for bridgeless maps:
in particular, it is easy to check (by simple inspection of Table~\ref{tab:smallex}) that the coefficient of $z^n u^k$ in
\begin{equation*}
C(z,u) = u + zu^2 + z^2(2u^2 + 2u^3) + z^3(10u^2 + 12u^3 + 5u^4) + \dots
\end{equation*}
does \emph{not} give the number of bridgeless maps with $n$ edges and $k$ vertices.
This can also be seen as an explanation for why the OMR bijection \emph{cannot possibly restrict} to a bijection between bridgeless maps and connected diagrams.
Indeed, we have the following somewhat curious situation:
\begin{enumerate}
\item The OMR bijection sends vertices to top chords, but does not restrict to a bijection between bridgeless maps and connected diagrams.
\item The $\phi$ bijection of Section~\ref{sec:bijection} restricts to a bijection between bridgeless maps and connected diagrams, but sends vertices to terminal chords rather than to top chords (see Proposition~\ref{prop:translation}).
\end{enumerate}
Taking either the $\phi$ bijection or the OMR bijection as primary leads naturally to two different open questions:
\begin{question}Is there a natural invariant $Q$ of maps, such that the coefficient of $z^n u^k$ in \eqref{eq:connected-simple} counts bridgeless maps with $n$ edges and $Q = k$?
\label{question1}
\end{question}
\begin{question}
Is there a natural property $P$ of maps, such that the coefficient of $z^n u^k$ in \eqref{eq:connected-simple} counts $P$-maps with $n$ edges and $k$ vertices?
\label{question2}
\end{question}

Furthermore, we can state at this point another interesting phenomenon. Combining Observation~1 and~2 from above shows that the number of indecomposable diagrams with $n$ chords and $k$ terminal chords is equal to the number of indecomposable diagrams with $n$ chords and $k$ top chords. This remarkable property, which had not been observed before, is another nice consequence of the existence of the $\phi$ bijection. In actual fact, the statistics counting terminal chords and top chords are more than equidistributed for indecomposable diagrams; they are symmetric:

\begin{proposition}[\cite{Lpersonal}] \label{prop:topterm-symmetry} Indecomposable diagrams of size $n$ with $k_1$ terminal chords and $k_2$ top chords are in bijection with indecomposable diagrams of size $n$ with $k_2$ terminal chords and $k_1$ top chords.
\end{proposition}

The proof of this result, which was communicated to the authors by Mathias Lepoutre~\cite{Lpersonal}, uses the fact that one can recursively change the position of the leftmost closing endpoint.

%% This question has been previously considered by Arquès and Micheli \cite{AMgde}, who derived Equation~\eqref{eq:ABeq1} by introducing certain topological operations on maps of arbitrary genus\footnote{Thanks to Maciej Do\l{}ega for pointing us to this reference.}.
%% Here we give a new bijective interpretation directly on indecomposable diagrams, with the crucial property that it restricts naturally to connected diagrams.

%%%%%%%%%%%%%%%%%%%%%%%%%%%%%%
\subsection{Combinatorial interpretation}
\label{subsec:decomp}
%%%%%%%%%%%%%%%%%%%%%%%%%%%%%%

Before describing the interpretation of Equations \eqref{eq:ABeq1} and \eqref{eq:connected-simple} on chord diagrams, we take the opportunity of refining them to keep track of the number of crossings.
\begin{theorem}\label{theo:refined-ab-equation}
Let $B(z,u,v)$ be the ordinary generating function of indecomposable diagrams counted with respect to the number of chords minus one ($z$), the number of top chords ($u$) and the number of crossings ($v$). Similarly, let $C(z,u,v)$ be the generating function for connected diagrams with the same interpretation of the parameters.
The following equations hold:
\begin{align}
B(z,u,v) &= u + zB(z,1+uv,v)\,B(z,u,v), \label{eq:indecomp} \\
C(z,u,v) &= u + z(C(z,1+uv,v) - C(z,1,v))\,C(z,u,v). \label{eq:connected}
\end{align}\label{theo:new-arques-beraud}
\end{theorem}
%\noindent
%\nz{the following text from the short version of the article might need a bit of editing for the long version}
\begin{proof}
%For the sake of simplicity, let us forget the variable $v$ counting the number of crossings (it can be incorporated quite simply afterwards), leaving just the equation $B(z,u) = u + zB(z,1+u)B(z,u)$.
%original (\ref{eq:ABeq1}).
%Equation~\eqref{eq:indecomp} then reads $B(z,u) = zu + B(z,1+u)B(z,u)$.
From a combinatorial point of view, Equation \eqref{eq:ABeq1} says that every indecomposable diagram with a least two chords can be seen as the product of two indecomposable diagrams, one of which has a marked subset of top chords.

We start by describing the combination part, building a diagram from two smaller ones. Figure~\ref{fig:newdecomposition} gives an example of such a combination.

\begin{figure}[!ht]
\begin{center}
\includegraphics[width=0.95\textwidth]{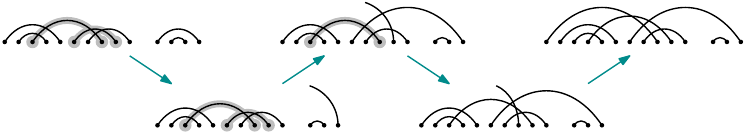}
\end{center}
\caption{An example of how to combine an indecomposable diagram with another indecomposable diagram in which a subset of top chords is marked. The first diagram has 4 top chords, 2 of which are marked. The second diagram has only one top chord. The combination of both induces 3 top chords, as expected.}
\label{fig:newdecomposition}
\end{figure}

Let us thus consider two indecomposable diagrams $D_1$ and  $D_2$, where some top chords of $D_1$ are marked. We run the following algorithm:
\begin{enumerate}
\item Put $D_2$ on the right of $D_1$. 
\item Open the left endpoint of the root chord $D_2$.
\item Consider the rightmost marked top chord. (The top chords are sorted from left to right without ambiguity.) If there are no more marked top chords, go to~\ref{item:nochord}. \label{item:nextchord}
\item Forget the marking of that chord. Then, open its left endpoint, and replace it by the left endpoint of the other open arc. Go to \ref{item:nextchord}. \label{item:swap}
\item Close the open arc at the left of $D_1$. \label{item:nochord}
\end{enumerate}
The composition of two diagrams is thus defined. We denote by $D$ the resulting diagram.

Let us enumerate the top chords of $D$. Each non-marked top chord of $D_1$ is now below a chord (which corresponds to the most immediate marked top chord to its right -- or the root chord of $D_2$ if there were not any marked top chords on its right), so is not a top chord in $D$ anymore. 
On the contrary, each marked top chord of $D_1$ remains a top chord. Indeed, the only chords that change from $D_1$ to $D$ are the marked top chords, and the algorithm is constructed in such a way that a marked top chord never covers the marked top chords on its left. As for the other chords of $D$, it only takes a quick check to observe that non-top chords stay non-top chords, and top chords of $D_2$ stay top chords. Finally, the top chords of $D$ are given by the top chords of $D_2$ and the marked top chords of $D_1$.

As for the number of crossings in $D$, we can notice that the algorithm only creates crossings during the execution of step \ref{item:swap}. Indeed, swapping an open arc and the left endpoint of a top chord (being on the left of the open arc) increases the number of crossings exactly by~$1$. That is why the number of crossings in $D$ is the number of the crossings of $D_1$ and $D_2$, plus the number of marked top chords.

We just proved that the multi-set of diagrams $D$ induced by the combinations of diagrams $D_1$ and $D_2$ has for generating function $zB(z,1+uv,v)\,B(z,u,v)$. To prove \eqref{eq:indecomp}, we only need to show that our way of combining two diagrams to produce a larger diagram is bijective. For the inverse operation, we run the following algorithm, starting from an indecomposable diagram $D$ of size $>1$.
\begin{enumerate}
\item Open the left endpoint of the root chord of $D$.
\item If the resulting diagram is not indecomposable, go to \ref{item:end}. \label{item:if}
\item Consider the leftmost top chord intersecting the open arc.
\item Open its left endpoint, and replace it by the left endpoint of the other open arc.
\item Mark the chord that was just closed. Go to \ref{item:if}.
\item Close the open arc to the right of the leftmost indecomposable component of $D$. We thus obtain two indecomposable diagrams $D_1$ (on the left) and $D_2$ (on the right). \label{item:end}
\end{enumerate}
To see that this algorithm computes an inverse to the first algorithm, the reader may refer again to Figure~\ref{fig:newdecomposition}, which can be likewise read from right to left.
%We thereby prove that the combination of two indecomposable diagrams is bijective, and so Equation~\eqref{eq:indecomp} holds.
This establishes that every indecomposable diagram which is not the one-chord diagram can be expressed as the combination of two diagrams, and so Equation~\eqref{eq:indecomp} holds.

Note that a new connected component is created by this process if and only if no top chord is marked. Indeed, the only way to form a new component is to close the root chord of $D_2$ directly at the left of $D_1$, which can be done by jumping  Item~\ref{item:swap}. So if we want to enumerate connected diagrams, we have to force diagrams $D_1$ to have at least one marked top edge. Such diagrams are counted by $C(z,1+uv,v) - C(z,1,v)$. We recover Equation~\eqref{eq:connected}.
\end{proof}

\subsection{An application to lambda calculus}
\label{subsec:lambda}
%%%%%%%%%%%%%%%%%%%%%%%%%%%%%%

The results of the previous sections have a surprising application to the combinatorics of lambda calculus.
As one of the authors described in \cite{Zcounting}, the original Arquès-Béraud equation \eqref{eq:ABeq1} is also satisfied by the generating function counting certain natural isomorphism classes of terms in lambda calculus (namely, \emph{neutral linear} terms modulo exchange of adjacent lambdas) by size and number of free variables.
This fits a broader pattern of combinatorial connections recently discovered between different fragments of lambda calculus and different families of maps, beginning with a bijection between rooted trivalent maps and linear lambda terms found by Bodini, Gardy, and Jacquot \cite{BoGaJa2013}, and a bijection between rooted planar maps and neutral planar lambda terms found by Giorgetti and Zeilberger \cite{ZGcorr}.
It was also shown in \cite{Ztrivalent} that the bijection of \cite{BoGaJa2013} restricts to a bijection between bridgeless (respectively, bridgeless planar) trivalent maps and linear (respectively, planar) lambda terms with no closed subterms -- such terms were called ``indecomposable'' in \cite{Ztrivalent}, but here we call them \emph{unit-free} to avoid confusion with indecomposable chord diagrams.
Similarly, it is not difficult to check that the bijection of \cite{ZGcorr} restricts to a bijection between bridgeless planar maps and unit-free neutral planar terms.
It is therefore tempting to draw the list of correspondences between families of lambda terms and families of rooted maps pictured in Table~\ref{table:lambda}, where on the right we have indicated the index for the relevant OEIS entry counting objects by size (note that the size of a 3-valent map is defined here as its number of vertices, rather than edges).

\begin{table}

\begin{center}
\begin{tabular}{llc}
family of lambda terms & family of rooted maps & OEIS entry \\
\hline
linear terms  & 3-valent combinatorial maps & A062980\\
planar terms  & planar 3-valent maps & A002005 \\
unit-free linear terms  & bridgeless 3-valent maps & A267827 \\
unit-free planar terms  & bridgeless planar 3-valent maps & A000309 \\
neutral linear terms/$\sim$  & combinatorial maps & A000698 \\
neutral planar terms  & planar maps & A000168 \\
{\bf unit-free neutral linear/$\sim$} & {\bf bridgeless maps} & {\bf A000699} \\
unit-free neutral planar  & bridgeless planar maps & A000260
\end{tabular}
\end{center}
\caption{ Known correspondences between different families of lambda terms and rooted maps. The correspondence in the boldfaced row was previously only conjectured, but is a corollary of our results here. }
\label{table:lambda}
\end{table}

The aforementioned works establish (either directly or as easy consequences) that each family of lambda terms is in the same combinatorial class as the corresponding family of rooted maps, \emph{for every row of Table~\ref{table:lambda} other than the boldfaced one.}
On the other hand, Proposition~\ref{prop:recurrence} above establishes that bridgeless maps are indeed counted by OEIS sequence A000699.
So, to verify the full table, all that remains is to show that unit-free neutral linear terms (modulo exchange of adjacent lambdas) are counted by the same sequence.
\begin{proposition}[{cf.~\cite{Zcounting,Ztrivalent}}]
\label{prop:lambda}
Let $C(z,u)$ be the two-variable generating function counting isomorphism classes of unit-free neutral linear lambda terms by size and number of free variables.
Then $C(z,u)$ satisfies equation \eqref{eq:connected-simple}.
\end{proposition}
\begin{proof}
This is essentially immediate from definitions: see the references \cite{Zcounting} and \cite{Ztrivalent} for formal definitions of the relevant terms, as well as for the proofs of very similar equations.
\end{proof}
\begin{corollary}
\label{corr:lambda-diagrams}
Isomorphism classes of unit-free neutral linear lambda terms of size $n$ and with $k$ free variables are equinumerous with connected chord diagrams of size $n$ and with $k$ top chords.
\end{corollary}
\begin{proof}
Since by Proposition~\ref{prop:lambda} and Theorem~\ref{theo:refined-ab-equation}, their generating functions both satisfy the same equation \eqref{eq:connected-simple}.
\end{proof}
\begin{corollary}\label{cor:lambda-maps}
The number of isomorphism classes of unit-free neutral linear lambda terms of size $n$ is equal to the number of rooted bridgeless combinatorial maps of size $n$.
\end{corollary}
\begin{proof}
By combining Corollary~\ref{corr:lambda-diagrams} with Theorem~\ref{theo:simplebijection} (or Proposition~\ref{prop:recurrence}).
\end{proof}
\noindent
It is worth remarking that our proof of this enumerative result also implicitly yields a bijection between isomorphism classes of unit-free neutral linear lambda terms and rooted bridgeless combinatorial maps, by composing the bijection $\theta$ between bridgeless maps and connected diagrams with the implicit bijection between connected diagrams and this family of lambda terms that results from their admitting the same recursive decomposition \eqref{eq:connected-simple}.
However, the meaning of this bijection is far less clear because we run into the obstacle posed by Question~\ref{question1}, namely, that it is not obvious what part of a rooted map should correspond to the free variables in a unit-free neutral linear term (i.e., what's counted by the $u$ parameter in $C(z,u)$).
On the other hand, one might try to side-step this obstacle by passing directly from (bridgeless) combinatorial maps to (unit-free) neutral linear terms via an analogue of the bijection of Section~\ref{sec:bijection}.
Given what we know about the transfer of statistics across that bijection (see Table~\ref{tab:transfer}), the following is therefore a natural related question.
\begin{question}
What (if anything) is the lambda calculus analogue for the terminal chords of a chord diagram?
In particular, is there a natural invariant $Q$ of neutral linear terms, such that there is a bijection between connected diagrams of size $n$ with $k_1$ top chords and $k_2$ terminal chords, and isomorphism classes of unit-free neutral linear terms of size $n$ with $k_1$ free variables and $Q = k_2$ (cf.~Corollary~\ref{corr:lambda-diagrams})?
(A good notion of $Q$ should also be symmetrically distributed with the number of free variables among neutral linear terms of size $n$, following Proposition~\ref{prop:topterm-symmetry}.)
\label{question3}
\end{question}
\noindent

% \begin{center}
% \begin{tabular}{lll}
% family of lambda terms & functional eqn.~for $T(z,u)$ & OEIS\\
% \hline
% linear terms & $u + zT(z,u)^2 + z\frac{\partial}{\partial u}T(z,u)$ & A062980\\
% planar terms & $u + zT(z,u)^2 + z\frac{T(z,u)-T(z,0)}{u}$ & A002005 \\
% unit-free linear & $u + z(T(z,u)-T(z,0))^2 + z\frac{\partial}{\partial u}T(z,u)$ & A267827 \\
% unit-free planar & $u + z(T(z,u)-T(z,0))^2 + z\frac{T(z,u) - T(z,0)}{u}$ & A000309 \\
% neutral linear terms/$\sim$ & $u + zT(z,u)T(z,u+1)$ & A000698 \\
% neutral planar terms & & A000168 \\
% neutral unit-free linear/$\sim$ & $u + zT(z,u)(T(z,u+1) - T(z,1))$ & A000699 \\
% neutral unit-free planar &  & A000260
% \end{tabular}
% \end{center}

\section{Conclusion}

After noticing an enumerative link between connected chord diagrams and bridgeless combinatorial maps, we made this observation into a size-preserving bijection $\theta$ by proving that these two families admit parallel decompositions in terms of a boxed product.  % (Definition~\ref{def:theta}) (Definitions~\ref{def:prod} and~\ref{def:prod2}) and its $\oplus$-variant (Definition~\ref{def var box})
An alternative decomposition based on root chord/root edge deletion then yielded another bijection $\phi$ between the larger families of indecomposable chord diagrams and rooted combinatorial maps, % (Definition~\ref{def:phi2})
but these two bijections turned out to be essentially equivalent: % (Proposition~\ref{prop:phi=theta})
$\theta$ is the restriction of $\phi$, while $\phi$ is the extension of $\theta$ obtained by composing with a ``connec\-ted/brid\-ge\-less root component'' decomposition.
% Notably, this property of $\phi = \overline\theta$ stands in contrast to an earlier bijection between indecomposable diagrams and maps given by Ossona de Mendez and Rosenstiehl, which does not restrict to connected diagrams and bridgeless maps.
Moreover, we established that the bijection $\phi = \overline\theta$ has many other interesting properties as well: vertices correspond to terminal chords; planarity is equivalent to a forbidden pattern in the world of diagrams.

% We gave a new bijection between indecomposable rooted chord diagrams and rooted combinatorial maps which restricts to a bijection between connected diagrams and bridgeless maps.

% In order to prove that these families of objects are in bijection, we have presented numerous decompositions that both classes share: decomposition according to the root chord/root edge (Definition~\ref{def:phi2}), $\star$-decomposition for connected diagrams and bridgeless maps (Definitions~\ref{def:prod} and~\ref{def:prod2}) and its $\oplus$-variant (Definition~\ref{def var box}), decomposition with respect to the ``connec\-ted/brid\-ge\-less root component" (Proposition~\ref{prop:decomposition}).
% We have noticed that the bijections\footnote{The plural is used but they are essentially one, as proved by Proposition~\ref{prop:phi=theta}.} resulting from  these decompositions inherit a lot of interesting properties (maybe more than Ossona de Mendez and Rosenstiehl's non-recursive bijection \cite{OdMRencoding}): bridgeless maps are mapped to connected diagrams; vertices correspond to terminal chords; planarity is equivalent to a forbidden pattern in the world of diagrams.

Some decompositions are apparently only meaningful for one of the two families, such as the decomposition of maps with respect to the number of ingoing edges for the rightmost DFS (Proposition~\ref{prop:translatedcrazyformula}), or the decomposition of diagrams with respect to the top chords (Theorem~\ref{theo:new-arques-beraud}).
On the other hand, since each of these decompositions describes interesting features for one of the combinatorial families, it is natural to wonder if they have analogues in the other class, highlighting new parameters (cf.~Questions~\ref{question1} and~\ref{question2}).
%(For example, this is precisely the purpose of open Questions~\ref{question1} and~\ref{question2}.) % which should require further investigation.
%
There are other nice consequences of the present work which concern transversal areas, such as quantum field theory or lambda calculus.
Indeed, our bijection between maps and diagrams has given interesting new perspectives on these domains and enabled a better understanding of some aspects of the theory, while suggesting a few natural directions for future work.
% However, not every mystery has been yet solved, notably in quantum field theory, but this should be the topic of a future work.

Finally, one may wonder about a non-recursive approach to a bijection between bridgeless maps and connected diagrams. Although the authors have thought in this direction and see no straightforward answer, it is not impossible that maps and diagrams conceal other nice connections.

\bibliographystyle{plain}
\bibliography{CoYeZe}

\end{document}